\documentclass[11pt,a4paper]{article}

\usepackage[english]{babel}
\usepackage{amsfonts,a4wide,bm,url}
\usepackage{amsmath}
\usepackage{amssymb}
\usepackage{amsthm}
\usepackage{stmaryrd}
\usepackage{hyperref}

\newtheorem{theorem}{Theorem}
\newtheorem*{theorem*}{Theorem}
\newtheorem{corollary}[theorem]{Corollary}
\newtheorem{lemma}[theorem]{Lemma}
\newtheorem{proposition}[theorem]{Proposition}
\newtheorem{definition}[theorem]{Definition}
\newtheorem{notation}[theorem]{Notation}
\newtheorem{remark}[theorem]{Remark}
\newtheorem{example}[theorem]{Example}

\def\N{\mathbb{N}}

\def\cD{\mathcal{D}}

\def\cZ{\mathcal{Z}}
\def\cM{\mathcal{M}}

\newcommand{\house}[1]{\widehat{#1}}

\newcounter{tmpabcd}

\newcounter{tmpnum}

\newcounter{tmprome}

\newcommand{\ep}{$\blacksquare$}

\newcommand{\mcc}{\mathbb{C}}

\newcommand{\mpp}{\mathbb{P}}

\newcommand{\mnn}{\mathbb{N}}

\newcommand{\mrr}{\mathbb{R}}

\newcommand{\mzz}{\mathbb{Z}}

\newcommand{\mqq}{\mathbb{Q}}

\newcommand{\eqdef}{\ensuremath{\stackrel{\mathrm{def}}{=}}}

\newcommand{\Dist}{\ensuremath{{\rm Dist}}}

\renewcommand{\b}[1]{{{#1}}}

\newcommand{\hidden}[1]{}

\newcommand{\Q}{\mathbb{Q}}
\newcommand{\R}{\mathbb{R}}
\newcommand{\C}{\mathbb{C}}

\newcommand{\ee}{{\rm e}}

\newcommand{\V}{\ensuremath{\mathcal{V}}}

\newcommand{\A}{\ensuremath{\mathcal{A}}}
\newcommand{\AnneauDePolynomes}{\A}

\newcommand{\Dc}{\mathcal{D}}

\newcommand{\ul}[1]{\underline{#1}}
\newcommand{\ull}[1]{\underline{\b{#1}}}

\newcommand{\ol}[1]{\overline{#1}}

\newcommand{\ord}{\ensuremath{{\rm ord}}}
\newcommand{\ordz}{\ensuremath{{\rm ord_{\b{z}=0}}}}

\newcommand{\trdeg}{\ensuremath{{\rm tr.deg.}}}

\newcommand{\card}{{\rm card\;}}

\begin{document}


\title{Algebraic independence and normality of the values of Mahler functions\thanks{The work on this project has been partially supported by EPSRC  Grant EP/M021858/1.}}

\author{Evgeniy Zorin\footnote{E. Zorin, University of York,   Department of Mathematics, York, YO10 5DD, United Kingdom, \href{mailto:evgeniy.zorin@york.ac.uk}{evgeniy.zorin@york.ac.uk}}}

\hidden{
\footnote{E. Zorin,
              University of York,   Department of Mathematics, York, YO10 5DD,  United Kingdom,
              Tel.: +44 1904 32 3080,
              \href{mailto:evgeniy.zorin@york.ac.uk}{evgeniy.zorin@york.ac.uk}
}
}


\maketitle

\begin{abstract}
The main purpose of this article is to provide new results on algebraic independence of values of Mahler functions and their generalizations. At the same time, we  establish new measures of algebraic independence for these values. In particular, we provide a measure of algebraic independence for values of Mahler functions at complex transcendental points. 

As an example of application of our new measures of algebraic independence, we are able to infer that a Mahler number does not belong to the class $U$ in Mahler's classification. Also, our results imply new examples, for $n\geq 1$ arbitrarily large, of sets $\left(\theta_1,\dots,\theta_n\right)\in\mrr^n$ normal in the sense of G.~Chudnovsky (1980).
\end{abstract}


\section{Introduction}


In this paper we establish a variety of new measures of algebraic independence of Mahler numbers, and at the same time we largely expand the limits of Mahler's method by treating the functional systems of the form
\hidden{
we consider systems of functions
\begin{equation} \label{intro_f}
f_1(\b{z}),...,f_n(\b{z})\in\ol{\mqq}[[\b{z}]]
\end{equation}
analytic in some neighbourhood $U$ of $0$ 
and satisfying the following system of functional equations:
}
\begin{equation} \label{systeme_1}
    a(\b{z})\ull{f}(\b{z})=A(\b{z})\ull{f}(p(\b{z}))+B(\b{z}),
\end{equation}
 where $p(\b{z})=p_1(\b{z})/p_2(\b{z})$ is a rational function with coefficients in $\ol{\mqq}$ and with the order of vanishing at $0$ at least 2, that is $\ordz p:=\ordz p_1-\ordz p_2\geq 2$. Also, $\ull{f}(\b{z})=(f_1(\b{z}),\dots,f_n(\b{z}))$ denotes an $n$-tuple of functions analytic in some neighbourhood $U$ of $0$ and having algebraic coefficients, $a(\b{z})\in\ol{\mqq}[\b{z}]$, $A$ (resp. $B$) is an $n\times n$ (resp. $n\times 1$) matrix with coefficients in $\ol{\mqq}[\b{z}]$. We assume in all this article that $\det A(\b{z})$ is a non-zero polynomial.

\hidden{
It is common to call \emph{Mahler functions} analytic functions 
\begin{equation} \label{intro_f}
f_1(\b{z}),...,f_n(\b{z})\in\ol{\mqq}[[\b{z}]]
\end{equation}
satisfying functional system~\eqref{systeme_1} with $p(z)=z^d$.
}
The classical case of \emph{Mahler functions} appear by setting $p(z)=z^d$ in the system~\eqref{systeme_1}.
The name is after Kurt Mahler, who initiated their study introducing~\emph{Mahler's method}~\cite{Mah1929,Mah1930,Mah1930_2}.
\hidden{
Algebraic independence of values of these functions also was studied by Amou, Becker, Kubota, Loxton, Masser, Nishioka, van der Poorten, T\"opfer and many others~\cite{A1991,B1994,Kubota1977,LvdP1982,LvdP1988,Ni1986,Ni1996,Pellarin2010,ThTopfer1995}.
}
\emph{Mahler numbers} are the numbers that can be presented as $f_1(\alpha)$, where $0<|\alpha|<1$ is an algebraic number and $f_1(z)$ is a Mahler function.

More generally, we refer to the solutions
\begin{equation} \label{intro_f}
f_1(\b{z}),...,f_n(\b{z})\in\ol{\mqq}[[\b{z}]]
\end{equation}
of system~\eqref{systeme_1} where $p(z)$ is not necessarily of the form $z^d$ as \emph{generalized Mahler functions}, and their values at algebraic points within their domain of convergence as \emph{generalized Mahler numbers}.

Transcendence and algebraic independence of Mahler numbers attracted a lot of interest and were intensively studied, among the others by Amou, Becker, Bell, Bugeaud, Coons, Kubota, Loxton, Masser, Nishioka, van der Poorten, T\"opfer and many others~\cite{A1991,B1994,BBC2013,Kubota1977,LvdP1982,LvdP1988,Ni1986,Ni1996,Pellarin2010,ThTopfer1995}.

At the same time, the case of generalized Mahler numbers did not appeared previously in the literature, at least as far as the author is aware of.

The interest in Mahler functions and Mahler numbers is manifold. First of all, this is a very important branch of study of algebraic independence theory.
Also, it has direct applications to theoretical computer science. For example, the set of Mahler numbers contains as a proper subset the set of \emph{automatic numbers} \cite{Cobham1968}, the numbers of the form $\sum_{n=0}^{\infty}a_n b^n$, $b\in\mzz_{\geq 2}$ where the sequence of digits $(a_n)_{n\in\mnn}$, $a_n\in\{0,\dots, b-1\}$, $n\in\mnn$ can be generated by a \emph{finite automaton} (the simplest class of Turing machines) which receives to the entry the $b$-adic expansion of the number $n$. We refer the reader to the book~\cite{AS2003} for much more detailed discussion.

Theorems (see e.g.~\cite{Cobham1968}) which provide a system of Mahler's functions~\eqref{intro_f} that specialize to an automatic number (i.e. $f_1(1/b)$, $b\in\mzz_{\geq 2}$, is an automatic number) always give a system of algebraically dependent functions. So if we keep in mind applications to automatic numbers, it is important to provide algebraic independence results for the values of solutions of~\eqref{systeme_1} in absence of hypothesis on algebraic independence of the functions~\eqref{intro_f} themselves. In this our article we focus on the measure of algebraic independence of values of generalized Mahler functions which may be algebraically dependent (this is in contrast to all the previous results on generalized Mahler functions~\cite{ThTopfer1995,EZ2010,EZ2013_2}). 

Given a solution~\eqref{intro_f} to the system~\eqref{systeme_1}, we denote by $t=t(\ul{f})$ the transcendence degree\footnote{Note that in most of statements of this article we deal with a fixed set of functions~\eqref{intro_f}, so, to simplify the notation, we most often use the notation $t$ rather than $t(\ul{f})$.}
\begin{equation} \label{intro_trdeg}
t:=\trdeg_{\mcc(z)}\mcc(f_1(z),\dots,f_n(z)),
\end{equation}
that is $t$ is the maximal number of functions among $f_1(z),\dots,f_n(z)$ which are algebraically independent over $\mcc(z)$. Up to reindexing $f_i$, $i=1,\dots,n$, we can assume that $f_1(z),\dots,f_t(z)$ are algebraically independent (hence the functions $f_{t+1}(z),\dots,f_n(z)$ are algebraic over $\mcc\left(z,f_1(z),\dots,f_t(z)\right)$). In all this article we assume $t\geq 1$.

In this article we provide, for $\gamma\in\mcc$, lower bounds for the transcendence degree
\begin{equation} \label{intro_trdeg_numbers}
\trdeg_{\mqq}\mqq(\gamma,f_1(\gamma),\dots,f_n(\gamma)),
\end{equation}
that is we estimate from below the number of algebraically independent numbers among $\gamma,f_1(\gamma),\dots,f_n(\gamma)$. Moreover, we provide a geometric refinement of lower bounds for~\eqref{intro_trdeg_numbers}. Indeed, the fact that
$$
\trdeg_{\mqq}\mqq(\gamma,f_1(\gamma),\dots,f_n(\gamma))\geq k,
$$
$k\in\mnn$, means that the corresponding point in the projective space
\begin{equation} \label{intro_point}
(1:\gamma:f_1(\gamma):\dots:f_n(\gamma))\in\mpp^{n+1}
\end{equation}
does not belong to any subvariety of $\mpp^{n+1}$ defined over $\ol{\mqq}$ and of dimension $\leq k-1$. In this article, we prove that the transcendence degree~\eqref{intro_trdeg_numbers} is at least $k$ by establishing a strictly positive lower bound for the distance from the point~\eqref{intro_point} to any subvariety $W$ of $\mpp^{n+1}$ defined over $\ol{\mqq}$ of dimension $\leq k-1$. Here we understand the distance in the sense of the \emph{projective distance}, as defined in~\cite[Chapter~6, \S~5]{NP}. We refer the reader to this reference for the general definition and detailed discussion of the properties. 
For instance, if $W$ is a zero locus of a homogeneous polynomial $P$, then the projective distance from a point $x\in\mpp^{n+1}$ to $W\subset\mpp^{n+1}$ can be substituted by the \emph{normalized value} of the polynomial $P$ at $x$, that is $\frac{|P(x')|}{|P|\cdot|x'|^{\deg P}}$, where $x'$ stands for any representative of the projective point $x$, $|P(x')|$ denotes the absolute value of $P(x')$ (archimedean or not), $|P|$ denotes the maximum of absolute values of coefficients of $P$ and $|x'|$ is the maximum of absolute values of coordinates of $x'$. See also Remark~\ref{remak_Dist} below for some more discussion of the projective distance $\Dist(x,W)$.

Naturally, such a lower bound depends on the degree of $W$ and on its height. We refer the reader to~\cite[Chapters~5 and~7]{NP} for the definition of the height and of degree of a projective variety. Here we remark only that if a projective variety $W$ is defined over $\mqq$, has codimension 1 and has no embedded components, that is
if $W$ is a zero locus of a homogeneous polynomial $P$ with integer coefficients, then the degree of $W$ coincides with the degree of the polynomial $P$ and the height of $W$ is the logarithmic Weil's height of $P$. We recall that logarithmic Weil's height of a polynomial $P$ with rational coefficients is defined by
$$
\sum_{v\in\cM}\log\left|P\right|_v,
$$
where $\cM$ is the set of all absolute values of $\mqq$ and $\left|P\right|_v$ denotes the maximum of the valuation $v$ of the coefficients of $P$. We also recall that if coefficients of $P$ are integers, then the exponential of logarithmic Weil's height of $P$ is comparable, up to a multiplicative constant that depends on the degree of $P$ only, to the naive height $H_{naive}(P)$ of the polynomial $P$, that is the maximum of (the archimedean) absolute values of its coefficients.

When 
$t=n$, $\gamma\in\ol{\mqq}$, and we consider only projective varieties of codimension 1 without embedded components,  our lower bound specializes to the classical \emph{measure of algebraic independence} of the numbers
\begin{equation} \label{intro_numbers_gamma}
\gamma, f_1(\gamma),\dots,f_n(\gamma),
\end{equation}
i.e. it can be interpreted as a lower bound for the values of non-zero polynomials $Q$ in $n+1$ variables with integer coefficients:
$$
\left|Q(\gamma, f_1(\gamma),\dots,f_n(\gamma))\right|\geq \phi(\deg(Q),h(Q)),
$$
where $|\cdot|$ denotes the archimedean absolute value and $\phi:\mnn\times\mrr^+\rightarrow\mrr^{+}$ is a function called a \emph{measure of algebraic independence} of numbers~\eqref{intro_numbers_gamma}. 

To prove our results we use a general method developed in~\cite{PP_KF} (see also~\cite{PP1997}). 
This method requires a multiplicity estimate, and recently a new result of this kind for solutions of~\eqref{systeme_1} was established  in \cite[Theorem~3.11]{EZ2010} and \cite{EZ2013}, 
see Theorem~\ref{theoNishioka} below. We use this new multiplicity estimate together with the general method from~\cite{PP_KF} to
improve previously known results and establish new facts on algebraic independence and measures of algebraic independence. 


We have also found an interesting application of our results to Diophantine approximations to a single Mahler number. To explain this application below, we recall first some definitions.

Recall that \emph{Mahler number}  is a number which admits a presentation as $f_1(\alpha)$ for $\alpha\in\mqq$, $0<|\alpha|<1$, where $f_1(z)$ is a component of a functional solution to the system~\eqref{systeme_1} with $p(z)=z^d$.

\emph{Lioville number} is an irrational number which admits, in a sense, infinitely many very nice approximations by rational numbers. More precisely, we say that a number $x$ is a \emph{Liouville number} if for every $n\in\N$ we can find a rational fraction $\frac{p_n}{q_n}\in\mqq$ such that
$$
0<\left|x-\frac{p_n}{q_n}\right|<\frac{1}{q_n^n}.
$$
It follows from Liouville's theorem that all Liouville numbers are transcendental. A classical explicit example by Liouville of a transcendental number
$
\sum_{k=0}^{\infty}10^{-k!}
$
is a Liouville number, hence the name of this class.

If we consider approximations of reals not only by rationals, but, more generally, by algebraic numbers, this will lead us to the idea of \emph{Mahler's classification}. This classification was first introduced by K.~Mahler~\cite{Mah1932}, and later J.~F.~Koksma~\cite{Kok1939} gave an alternative interpretation of the same classification. 

To introduce \emph{Mahler's classification}, consider the quantity $w_m(\xi)$ defined to be the supremum of $w\in\R_+$ such that the following inequality has infinitely many solutions in polynomials $P(X)\in\mzz[X]$ of degree at most $m\in\N$:
$$
0<\left|P(\xi)\right|<H_{naive}(P)^{-w},
$$
where the $H_{naive}(P)$ denotes the naive height of $P$ (recall that the naive height of the polynomial $P\in\mzz[X]$ is the maximum of the archimedean absolute value of its coefficients).
Note that Liouville numbers defined above is precisely the class of numbers with $w_1(\xi)=\infty$.

At the next step, define
$$
w(\xi):=\limsup_{m\rightarrow\infty}\frac{w_m(\xi)}{m}.
$$
It is easy to verify that $w(\xi)\geq 2$.

Mahler's classification~\cite{Mah1932} splits all the real numbers $\xi\in\R$ into classes according to the value of $w(\xi)$:
\begin{itemize}

\item the class of $A$-numbers is defined by $w(\xi)=2$,

\item the class of $S$-numbers is defined by $2<w(\xi)<\infty$,

\item the class of $T$-numbers is defined by $w(\xi)=\infty$ and $w_m(\xi)<\infty$ for all $m\in\N$.

\item the class of $U$-numbers is defined by $w(\xi)=\infty$ and $w_m(\xi)=\infty$ for some $m\in\N$.

\end{itemize}

It is known that the class of $A$-numbers coincides with $\ol{\mqq}$ and that the compliment to the class $S$ has Lebesgue measure 0.

The application we are going to discuss relies on our Theorems~\ref{annexe_theo_ia1} and~\ref{annexe_theo_2}, presented below in Section~\ref{section_Results}.
In case if in the system~\eqref{systeme_1} we have $\ordz p(z)=\deg p(z)$, then our measures of algebraic independence 
given in these theorems
are optimal in $h(W)$.
This allows us to deduce the following statement (see Section~\ref{section_U_numbers}, Theorem~\ref{corollary_U_numbers} for even more detailed result).

\medskip

\noindent {\bf Theorem.} A Mahler number can not be a $U$ number.

\medskip

The first result in this direction was established in~\cite{AC2006}, where it was proved that automatic numbers are not Liouville. 
Further, in \cite{AB2011} it was established that automatic numbers does not belong to the class $U$. Recently, these results has been improved~\cite{BBC2013} by showing that, under some conditions on $a,b\in\mzz$ (in particular, $b\geq 2$ and $\log|a|/\log b\in (0,1/3)$), the number $f_1(a/b)$ is not Liouville. Moreover, if $f_1(z)$ is a so called \emph{regular} series, than $f_1(a/b)$ does not belong to the class $U$ (the numbers of the form $f_1(1/b)$, $b\geq 2$, as described in this paragraph form a set that contains all the automatic numbers and itself is a proper subset of the set of Mahler numbers, see~\cite{AS2003}).


This our paper is organized as follows. We start by presenting our main results and their corollaries in Section~\ref{section_Results}. Section~\ref{section_Results} also contains, for illustrative purposes, a few of concrete examples on algebraic independence results that can be inferred from our general statements.

In Section~\ref{section_U_numbers} we deduce 
from our result on algebraic independence that Mahler numbers does not belong to the class $U$ in Mahler's classification.

In Section~\ref{section_cai} we provide a criterion for algebraic independence, which is a central tool in the proofs of our main results, Theorems~\ref{annexe_theo_ia1}, \ref{annexe_theo_ia2} and~\ref{annexe_theo_2}. This criterion, similarly to many other theorems of this kind, relies on the existence of a sequence of polynomials with nice approximation properties at the given point. Such polynomial sequences are constructed in Section~\ref{section_PS}. To this end we use a general  extrapolative construction from~\cite{PP_KF} (see Theorem~\ref{annexe_P7} below) presented in Section~\ref{section_K_functions}. To present this extrapolative construction, we need to remind the reader (a simplified version of) the notion of $K$-functions, which are needed to state Theorem~\ref{annexe_P7}. We do this in Section~\ref{section_K_functions} as well.

Finally, the proofs of our principal results, Theorems~\ref{annexe_theo_ia1}, \ref{annexe_theo_ia2} and~\ref{annexe_theo_2}, are given in Section~\ref{section_Proofs}.

\section{Notations and Results} \label{section_Results}

Throughout the text we use the following notation. For the rational fraction
$p(z)=p_1(z)/p_2(z)$, where $p_1(z),p_2(z)\in\ol{\mqq}[z]$ are coprime polynomials, we write
  $d := \deg p=\max\left(\deg p_1,\deg p_2\right)$, 
  $\delta := \ordz p = \ordz p_1 - \ordz p_2$.
For $h\in\mzz_{\geq 0}$, we will denote by $p^{[h]}(y)$ the $h$-th iterate of $p$ at a point $y\in\mcc$, i.e. we define recursively $p^{[0]}(y):=y$ and $p^{[h+1]}:=p(p^{[h]}(y))$ for all $h\in\mzz_{\geq 0}$.

Note that often we restrain our attention to the case of polynomial $p(z)\in\ol{\mqq}[z]$ (in this case, in the notation of the previous paragraph we have $p_2(z)=1$ and $p_1(z)=p(z)$). In Theorems~\ref{annexe_theo_ia1} and~\ref{annexe_theo_ia2} below we deal with the case of a polynomial $p(z)$, while Theorem~\ref{annexe_theo_2} treats the case if $p(z)$ is a rational fraction.

We will use the notions of \emph{degree}, \emph{height}, \emph{size} and \emph{distance to a point} of a projective varieties or a homogeneous ideals. These notions extend the corresponding characteristics of polynomials. We refer the reader to~\cite{NP}, Chapters~5-8 for the definitions and properties of degree $\deg(V)$ and height $h(V)$ of a projective variety $V\subset\mpp^n$, as well as for the definition and properties of the (projective) distance of $V$ to a point $x\in\mpp^n$, which is denoted by $\Dist(x,V)$.

The \emph{size} t(V) of a projective variety $V\subset\mpp^n$ is defined by
\begin{equation} \label{def_tV}
t(V):=h(V)+(\dim(V)+1)\deg(V)\log(n+1).
\end{equation}

Let $\alpha\in\ol{\mqq}$ be an algebraic number and let $|\cdot|$ denotes the archimedean absolute value on $\mcc$. 
We call the \emph{house of $\alpha$} the maximum
$$
\house{\alpha}:=\max_{\sigma:\ol{Q}\hookrightarrow\mcc}|\sigma(\alpha)|,
$$
where the maximum is taken over all the possible embeddings $\sigma$ of the field of algebraic numbers $\ol{\mqq}$ into $\mcc$.

Let $m,n\in\mnn$ and let $P\in\ol{\mqq}[X_1,\dots,X_n]=\sum_{0\leq m_1,\dots,m_n\leq m}a_{m_1,\dots,m_n}X_1^{m_1},\dots X_n^{m_n}$ be a polynomial in $n$ variables with algebraic coefficients, $a_{m_1,\dots,m_n}\in\ol{\mqq}$ for all $0\leq m_1,\dots,m_n\leq m$. We define the {\emph length} of the polynomial $P$ to be the sum of the houses of all its coefficients:
$$
L(P):=\sum_{0\leq m_1,\dots,m_n\leq m}\house{a}_{m_1,\dots,m_n}.
$$

Here is one of our main results on algebraic independence.

\begin{theorem} \label{annexe_theo_ia1}
Let $p({z})\in\ol{\mqq}[{z}]$ and let $f_1({z})$,\dots,$f_n({z})$ be functions analytic in a neighbourhood $U$ of 0 and satisfying~\eqref{systeme_1} as described in the beginning of this paper. Also, recall the notation~\eqref{intro_trdeg}. Let $y\in \ol{\mqq}\cap U$ be such that
\begin{equation*}
    p^{[h]}(y)\rightarrow 0
\end{equation*}
as $h\rightarrow\infty$ and no iterate $p^{[h]}(y)$ is a zero of $z\det A(z)$.


Then there is a constant $C>0$ such that for any variety $W\subset\mpp^n_{\mqq}$ of dimension $k < t+1-\frac{\log d}{\log\delta}$, one has
\begin{equation} \label{annexe_theo_ia1_result}
   \log\Dist(x,W)\geq-Ch(W)\left(\log h(W)\right)^{\frac{(t+1)\left(\frac{\log d}{\log\delta}-1\right)}{t-k+1-\frac{\log d}{\log\delta}}}\left(\deg(W)\right)^{\frac{t+1}{t-k+1-\frac{\log d}{\log\delta}}}
\end{equation}
where $x=\left(1:f_1(y):\dots:f_n(y)\right)\in\mpp^n_{\mcc}$.
\end{theorem}
\begin{remark} \label{remak_Dist}
The definition of $\Dist(x,W)$, for a point $x\in\mpp^n$ and a subvariety $W$of the same space, can be found in \cite[ Chapter~6, \S~5]{NP} or~\cite[\S~1.3]{EZ2010} (see~\cite[Definition~1.17]{EZ2010} and discussion after it). There are two simple cases which are, in a sense, the most important. First of all, if $W$ is a hypersurface defined by a homogeneous polynomial $P$, then $\Dist(x,W)$ is essentially $|P(x)|$ (more precisely, in this case $\log\Dist(x,W)=\log|P(x)|-\deg(P)\cdot\log|x|-\log|P|$, a so called \emph{normalized value of $P$ at $x$}, which gets rid of common factors in coordinates of a representative of the projective point $x$ and of the size of coefficients of $P$). So essentially if $k=n-1$ in Theorem~\ref{annexe_theo_ia1}, then one can, roughly speaking, substitute $\log|P(x)|$ in place of $\log\Dist(x,W)$ in the left hand side of~\eqref{annexe_theo_ia1_result}. 

On the other hand, for all points $x\in\mpp^n$ and all subvarieties $W\in\mpp^n$ one has $\Dist(x,W)=0$ iff $x\in W$. So if some value of $k$ is admitted in Theorem~\ref{annexe_theo_ia1} (i.e. if $k < t+1-\frac{\log d}{\log\delta}$), then at least $k+1$ of values $f_1(y),\dots,f_n(y)$ are algebraically independent over $\mqq$ (as the r.h.s. of~\eqref{annexe_theo_ia1_result} is $>-\infty$ in this case). Using this fact we readily deduce the following two corollaries.
\end{remark}
\begin{corollary} \label{annexe_cor_ia0}
Assuming the conditions of Theorem~\ref{annexe_theo_ia1} one has
\begin{equation*}
    \trdeg_{\mqq}\mqq\left(f_1(y),\dots,f_n(y)\right)\geq t+1-\left\lfloor\frac{\log d}{\log\delta}\right\rfloor,
\end{equation*}
where $\lfloor*\rfloor$ denotes the biggest integer less than or equal to $*$.
\end{corollary}
\begin{corollary} \label{annexe_cor_ia1}
Assuming the conditions of Theorem~\ref{annexe_theo_ia1} and $\frac{\log d}{\log\delta}<2$ one has
\begin{equation} \label{annexe_cor_ia1_est} \index{Degre de transcendance, minoration@Degr\'e de transcendance, minoration}
    \trdeg_{\mqq}\mqq\left(f_1(y),\dots,f_n(y)\right)=t.
\end{equation}
\end{corollary}
\begin{remark}
Corollary~\ref{annexe_cor_ia0} improves the lower bound
\begin{equation*}
    \trdeg_{\mqq}\mqq\left(f_1(y),\dots,f_n(y)\right)\geq \lceil(t+1)\frac{\log\delta}{\log d}-1\rceil,
\end{equation*}
established, in the case $n=t$ only, by in~\cite[Theorem~3]{ThTopfer1995}, where $\lceil*\rceil$ denotes the smallest integer bigger than or equal to $*$.

Corollary~\ref{annexe_cor_ia1} improves on~\cite[Corollary~2]{ThTopfer1995}, which gives only the case $n=1$ of~\eqref{annexe_cor_ia1_est}.
\end{remark}

We can also give a measure of algebraic independence of values $y,f_1(y),\dots,f_n(y)$, for an arbitrary $y\in\mcc^*$, which need not to be algebraic. This type of results for a transcendental $y$ has not been considered before, though our estimates in this situation are weaker than in the case of algebraic $y$.


\begin{theorem} \label{annexe_theo_ia2} Let $f_1({z})$,\dots,$f_n({z})$ be analytic functions as described in the beginning of this paper, and recall the notation~\eqref{intro_trdeg}. Assume that $p({z})\in\ol{\mqq}[{z}]$ with $\delta=\ord_{z=0} p({z})\geq 2$  and $d=\deg p(z)$.
Let $y\in U$ be such that
\begin{equation*}
    p^{[h]}(y)\rightarrow 0\text{ as }h\rightarrow\infty
\end{equation*}
and no iterate $p^{[h]}(y)\ne 0$ is a zero of $z\det A(z)$.

Then for all $\varepsilon>0$ there is a constant $C$ such that for any variety $W\subset\mpp^{n+1}_{\mqq}$ of dimension $k < t+1-2\frac{\log d}{\log\delta}$, one has
\begin{equation}  \label{annexe_theo_ia2_result}
   \log\Dist(x,W)\geq-C\max\left(\left(\deg(W)\right)^{\frac{t+2+\varepsilon}{t-k+1-2\frac{\log d}{\log\delta}-\varepsilon}},h(W)^{\frac{t+2+\varepsilon}{t-k+2-\frac{\log d}{\log\delta}-\varepsilon}}\right),
\end{equation}
where $x=\left(1:y:f_1(y):\dots:f_n(y)\right)\in\mpp^{n+1}_{\mcc}$.
\end{theorem}
One readily deduces two corollaries:
\begin{corollary} \label{annexe_cor_ia2}
Assuming the conditions of Theorem~\ref{annexe_theo_ia2} one has
\begin{equation*}
    \trdeg_{\mqq}\mqq\left(y,f_1(y),\dots,f_n(y)\right)\geq t+1-2\left\lfloor\frac{\log d}{\log\delta}\right\rfloor.
\end{equation*}
\end{corollary}

\smallskip

\begin{corollary} \label{annexe_cor_ia3}
Assuming the conditions of Theorem~\ref{annexe_theo_ia2} and $\frac{\log d}{\log\delta}<3/2$ one has
\begin{equation} \label{annexe_cor_ia3_est}
    \trdeg_{\mqq}\mqq\left(y,f_1(y),\dots,f_n(y)\right)\geq t-1.
\end{equation}
\end{corollary}

\begin{remark}
Corollary~\ref{annexe_cor_ia3} improves on the result~\cite{A1991}, which is applicable only to algebraically independent functions~\eqref{intro_f} solving system~\eqref{systeme_1} with $p(z)=z^d$ and than provides only the lower bound $\left\lfloor\frac{n+1}{2}\right\rfloor$.
\end{remark}

The next theorem improves Theorems~1 and~2 of~\cite{ThTopfer1995}, qualitatively et quantitatively.
\begin{theorem} \label{annexe_theo_2}
 Let $f_1({z})$,\dots,$f_n({z})$ be a collection of functions such as described in the beginning of this paper, and recall the notation~\eqref{intro_trdeg}. In this statement we assume $p(\b{z})\in\ol{\mqq}(\b{z})$ in the system~\eqref{systeme_1} and not only $p(\b{z})\in\ol{\mqq}[\b{z}]$ (compare with more restrictive assumption $p(\b{z})\in\ol{\mqq}[\b{z}]$ in preceding Theorems~\ref{annexe_theo_ia1} and~\ref{annexe_theo_ia2}). We keep the notation $d=\deg p$, $\delta=\ordz p\geq 2$. 
 Assume that $f_i(0)=0$, $i=1,\dots,n$, a number $y\in U\cap\ol{\mqq}$ satisfies $\lim_{h\rightarrow\infty}p^{[h]}(y)=0$ and for all $h\in\mnn$ the number $p^{[h]}(y)$ is not a zero of $z\det A(\b{z})$. Then there is a constant $C>0$ such that for any variety $W\subset\mpp^n_{\mqq}$ of dimension $k < t\left(2-\frac{\log d}{\log\delta}\right)$, one has
\begin{equation} \label{annexe_theo_2_result}
   \log\Dist(\ul{x},W)\geq -Ch(W)^{1+\frac{\log d-\log\delta}{(2t-k)\log\delta-t\log d}t}\deg(W)^{\frac{2t}{t\left(2-\frac{\log d}{\log\delta}\right)-k}},
\end{equation}
where $\ul{x}=\left(1:f_1(y):\dots:f_n(y)\right)\in\mpp^n_{\mcc}$.
In particular,
\begin{equation} \label{annexe_theo_2_trdeg} \index{Degre de transcendance, minoration@Degr\'e de transcendance, minoration}
    \trdeg_{\mqq}\mqq\left(f_1(y),\dots,f_n(y)\right)\geq  t\left(2-\frac{\log d}{\log\delta}\right).
\end{equation}
\end{theorem}

Now we give a family of concrete examples, with sets of functions of arbitrary size $n\geq 1$ and satisfying all the hypothesis of our theorems. We start with a particular case of~(\ref{systeme_1}) when this system has a simple diagonal form:
\begin{equation} \label{system_khi}
    \chi_i(z)=\chi_i\left(p(z)\right)+q_i(z), \quad i=1,\dots,n,
\end{equation}
where $p\in\ol{\mqq}(z)$ and $q_i\in\ol{\mqq}[z]$, $i=1,\dots,n$. Assuming $\deg q_i\geq 1$ and $q_i(0)=0$, $i=1,\dots,n$, $\ord_{z=0}p\geq 2$ we obtain solutions of~(\ref{system_khi}) analytic in some neighbourhood of 0,
\begin{equation} \label{khi_explicit}
    \chi_i(z)=\chi_i\left(p(z)\right)+q_i(z), \quad i=1,\dots,n.
\end{equation}
Lemma~\ref{annexe_chi_ind} below allows to verify the algebraic independence of $\chi_1,\dots,\chi_n$ over $\mcc(z)$. It is an easy corollary of \cite[Lemma~6]{ThTopfer1995} (as well as of \cite[Theorem~2]{Kubota1977}).
\begin{lemma} \label{annexe_chi_ind}
Let $n\in\mnn^*$, $q_i\in\mcc[\b{z}]$, $i=1,\dots,n$ and $p\in\mcc[\b{z}]$ satisfying $q_i(0)=0$, $i=1,\dots,n$ and $p(0)=0$.
Let $\chi_1,\dots,\chi_n\in\mcc((\b{z}))$ be functions defined by~(\ref{khi_explicit}). Suppose that $1,q_1,\dots,q_n$ are $\mcc$-linearly independent and at least one of the following conditions is satisfied:
\begin{enumerate}
  \item \label{annexe_chi_ind_condition_a} $\deg p \nmid \deg\left(\sum_{i=1}^n s_i q_i(\b{z})\right)$ for all $(s_1,\dots,s_n)\in\mcc^n\setminus\{\ul{0}\}$, \label{annexe_theo_khi_cond1}
  \item \label{annexe_chi_ind_condition_b} $\sum_{i=1}^n s_i \chi_i(\b{z}) \not\in \mcc[\b{z}]$ for all $(s_1,\dots,s_n)\in\mcc^n\setminus\{\ul{0}\}$. \label{annexe_theo_khi_cond2}
\end{enumerate}
Then the functions $\chi_1,\dots,\chi_n$ are algebraically independent over $\mcc(\b{z})$.
\end{lemma}
Using this lemma (especially point~(\ref{annexe_chi_ind_condition_a}) which is due to Th.T\"opfer) we can produce a large family of algebraically independent  sets of functions~(\ref{khi_explicit}). All these sets satisfy the hypothesis imposed on functions $f_1,\dots,f_n$ considered in this article, so we can apply Theorems~\ref{annexe_theo_ia1}, \ref{annexe_theo_ia2}, \ref{annexe_theo_2} and their corollaries to them substituting $n$ in place of $t=t(\ul{f})$.
\begin{theorem} \label{annexe_theo_khi_2} Let $p\in\ol{\mqq}(z)$, $d=\deg p$ and $\delta=\ord_{z=0}p\geq 2$. Let $n\in\mnn^*$, $q_i\in\ol{\mqq}[z]$, $q_i(0)=0$ ($i=1,\dots,n$) and let $\chi_i(z)$ be functions defined by~\eqref{system_khi}. Assume that $1,q_1,\dots,q_n$ are $\mcc$-linearly independent and that at least one of the following conditions is satisfied:
\begin{enumerate}
  \item $\deg p \nmid \deg\left(\sum_{i=1}^n s_i q_i(z)\right)$ for all $(s_1,\dots,s_n)\in\mcc^n\setminus\{\ul{0}\}$. \label{annexe_theo_khi_2_cond1_i}
  \item $\sum_{i=1}^n s_i \chi_i(z) \not\in \mcc[z]$ for all $(s_1,\dots,s_n)\in\mcc^n\setminus\{\ul{0}\}$. \label{annexe_theo_khi_2_cond2_i}
\end{enumerate}

Fix a $y\in\ol{\mqq}^*$ such that $p^{[h]}(y)\rightarrow 0$ as $h\rightarrow\infty$ and $p^{[h]}(y)\ne 0$ for all $h\in\mnn$. Then there exists a constant $C>0$ such that for every variety $W\subset\mpp^n_{\mqq}$ of dimension $k < n\left(2-\frac{\log d}{\log\delta}\right)$, one has the measure of algebraic independence~\eqref{annexe_theo_2_result} 
at $x=\left(1:\chi_1(y):\dots:\chi_n(y)\right)\in\mpp^n_{\mcc}$.

In particular,
\begin{equation*}
    \trdeg_{\mqq}\mqq\left(\chi_1(y),\dots,\chi_n(y)\right)\geq n\left(2-\frac{\log d}{\log\delta}\right).
\end{equation*}
\end{theorem}

\begin{theorem} \label{annexe_theo_khi} Let $p\in\ol{\mqq}[z]$, $d=\deg p$ and $\delta=\ord_{z=0}p\geq 2$. Let $n\in\mnn^*$, $q_i\in\ol{\mqq}[z]$, $q_i(0)=0$ ($i=1,\dots,n$) and let $\chi_i(z)$ be functions defined by~\eqref{system_khi}. Assume that $1,q_1,\dots,q_n$ are $\mcc$-linearly independent and that at least one of the following conditions is satisfied:
\begin{enumerate}
  \item $\deg p \nmid \deg\left(\sum_{i=1}^n s_i q_i(z)\right)$ for all $(s_1,\dots,s_n)\in\mcc^n\setminus\{\ul{0}\}$. \label{annexe_theo_khi_cond1_i}
  \item $\sum_{i=1}^n s_i \chi_i(z) \not\in \mcc[z]$ for all $(s_1,\dots,s_n)\in\mcc^n\setminus\{\ul{0}\}$. \label{annexe_theo_khi_cond2_i}
\end{enumerate}

Fix a $y\in\mcc^*$ such that
\begin{equation*} 
\lim_{h\rightarrow\infty}p^{[h]}(y)=0\text{ and } p^{[h]}(y)\ne 0\text{ for all }h\in\mnn.
\end{equation*}
Then the following holds true.
\begin{enumerate}
  \item \label{annexe_theo_khi_point_i} For every $\varepsilon>0$ there exists a constant $C$ such that for every projective variety $W\subset\mpp^{n+1}_{\mqq}$ of dimension $k < n+1-2\frac{\log d}{\log\delta}$, one has the measure of algebraic independence~\eqref{annexe_theo_ia2_result} at $x=\left(1:y:\chi_1(y):\dots:\chi_n(y)\right)\in\mpp^{n+1}_{\mcc}$. In particular,
\begin{equation*}
    \trdeg_{\mqq}\mqq\left(y,\chi_1(y),\dots,\chi_n(y)\right)\geq n+1-2\left\lfloor\frac{\log d}{\log\delta}\right\rfloor.
\end{equation*}
  \item \label{annexe_theo_khi_point_ii} Moreover, if $y\in\ol{\mqq}^*$, then 
there exists a constant $C>0$ such that for all projective variety $W\subset\mpp^n_{\mqq}$ of dimension $k < n+1-\frac{\log d}{\log\delta}$, one has the measure of algebraic independence~\eqref{annexe_theo_ia1_result} at $x=\left(1:\chi_1(y):\dots:\chi_n(y)\right)\in\mpp^n_{\mcc}$. In particular,
\begin{equation*}
    \trdeg_{\mqq}\mqq\left(\chi_1(y),\dots,\chi_n(y)\right)\geq n+1-\left\lfloor\frac{\log d}{\log\delta}\right\rfloor.
\end{equation*}
\end{enumerate}
\end{theorem}

\begin{example} \label{example_M} Consider the function that lies, in a sense, at the origin of Mahler's method: 
$$
M(z)=\sum_{k=0}^{\infty}z^{2^k}.
$$
Then, for all $y\in\mcc$, $0<|y|<1$,  in the infinite family of numbers
\begin{equation} \label{sequence_yM}
y,M(y), M(y^3), M(y^5),\dots,M(y^{2k+1}),\dots
\end{equation}
at most two numbers are algebraically dependent over $\mqq$.
To see this, apply Theorem~\ref{annexe_theo_khi} with $p(z)=z^{2}$, $q_i(z)=z^{2i+1}$, $i=1,2,3,\dots,n$, where $n\in\mnn$ is arbitrarily large. In this case, condition~\ref{annexe_theo_khi_2_cond1_i} of Theorem~\ref{annexe_theo_khi} is verified, hence functions $\chi_i(z)$ defined by~\eqref{system_khi} are algebraically independent. Note that in view of our choice of $p$ and $q_i$ we have $\chi_i(z)=M(z^i)$ for any $i\in\mnn$, hence we infer the claim from Theorem~\ref{annexe_theo_khi}, part~\ref{annexe_theo_khi_point_i}. Moreover, we find with part~\ref{annexe_theo_khi_point_ii} of the same theorem that in the case  if, in addition, $y\in\ol{\mqq}$, then all the numbers
\begin{equation} \label{sequence_M}
M(y), M(y^3), M(y^5),\dots,M(y^{2k+1}),\dots
\end{equation}
are algebraically independent. Note that the latter claim, on algebraic independence of numbers~\eqref{sequence_M} follows from the classical result by Nishioka~\cite{Ni1996}, whilst the example concerning numbers~\eqref{sequence_yM} was previously unknown, as far as the author knows.
\end{example}

\begin{example} \label{example_T}
We can produce a more exotic counterpart of Example~\ref{example_M} using Chebyshev's polynomials
$$
T_n=n\sum_{k=0}^n (-2)^k\frac{(n+k-1)!}{(n-k)!(2k)!}(1-x)^k, n\in\mnn.
$$
It is known that these polynomials commute under composition~\cite{Ritt1922},
$$
T_m(T_n)=T_n(T_m) \text{ for all } m,n\in\mnn.
$$
Using the same reasoning as above, we readily find with Theorem~\ref{annexe_theo_khi} the following result. Define function $\tau$ by
$$
\tau(z)=\sum_{k=1}^{\infty}T_2^{[k]}(z),
$$
where we denote by $T_2^{[k]}$ the $k$-th iteration of the second Chebyshev's polynomial $T_2$. For any $y\in\mcc\setminus\{0\}$ such that $T_2^{[k]}(y)\ne 0$ and
$$
\lim_{k\rightarrow\infty}T_2^{[k]}(y)=0
$$
among the infinite set of numbers
$$
y,\tau(T_3(y)),\tau(T_5(y)),\tau(T_7(y)),\dots,\tau(T_{2k+1}(y)),\dots
$$
at most three numbers are algebraically dependent over $\mqq$. If, in addition, $y\in\ol{\mqq}$, then all the numbers
$$
\tau(T_3(y)),\tau(T_5(y)),\tau(T_7(y)),\dots,\tau(T_{2k+1}(y)),\dots
$$
are algebraically independent.

Note that we cannot produce further examples with the same method because, by a classical result of Ritt~\cite{Ritt1922}, all the pairs of commuting polynomials,
$$
P(Q(z)=Q(P(z)),
$$
are, up to a linear homeomorphism, either both powers of $z$ (this case gives us Example~\ref{example_M}), or both Chebyshev's polynomials (this case gives us Example~\ref{example_T}), or iterates of the same polynomial (the latter case applied to the construction of Examples~\ref{example_M} and~\ref{example_T} produces, evidently, functions that all are algebraically dependent).
\end{example}
One more our example deals with so called Cantor series. These functions were introduced in~\cite{Ta1991} and studied further in~\cite{ThTopfer1995}. They are defined by
\begin{equation} \label{annexe_def_theta_i}
    \theta_i(z)=\sum_{h=0}^{\infty}\frac{1}{q_i(z)q_i(p(z))\cdots q_i(p^{[h]}(z))}, \quad (i=1,\dots,n),
\end{equation}
where $p(z)=p_1(z)/p_2(z)\in\ol{\mqq}(z)$, $\deg p_j=d_j$ ($j=1,2$), $\delta=\ord_{z=0} p\geq 2$, $q_i\in\ol{\mqq}[z]$ with $\deg q_i\geq 1$ and $|q_i(0)|>1$, $i=1,\dots,n$.

The functions $\theta_i$ ($i=1,\dots,n$) are analytic in a neighbourhood of 0 and satisfy the functional equation
\begin{equation*}
    \theta_i(p(z))=q_i(z)\theta_i(z)-1, \quad (i=1,\dots,n).
\end{equation*}
It is verified in~\cite{ThTopfer1995} that under some conditions, given in the statements of Theorems~\ref{annexe_theo_theta} and~\ref{annexe_theo_theta_2} below, these functions are algebraically independent.

If $p(z)$ is a polynomial, we are in measure to apply Theorems~\ref{annexe_theo_ia1} and~\ref{annexe_theo_ia2} to find the following results (compare with~\cite[Corollary~6]{ThTopfer1995})

\begin{theorem} \label{annexe_theo_theta} Let us consider different polynomials $q_1,\dots,q_n\in\ol{\mqq}[z]$ and let $p$ be in $\ol{\mqq}[z]$ such that $|q_i(0)|>1$ and $1\leq\deg q_i<d-1$, $i=1,\dots,n$, and $2<\deg p= d$. Define functions $\theta_i$, $i=1,\dots,n$, by~\eqref{annexe_def_theta_i}. Let $y\in\mcc^*$ satisfy $\lim_{h\rightarrow\infty}p^{[h]}(y)=0$ and $q_i\left(p^{[h]}(y)\right)\ne 0$ and $p^{[h]}(y)\ne 0$ for all $h\in\mnn$, $i=1,\dots,n$.
Then the following holds true.
\begin{enumerate}
  \item For every $\varepsilon>0$ there exists a constant $C>0$ such that for every projective variety $W\subset\mpp^{n+1}_{\mqq}$ of dimension $k < n+1-2\frac{\log d}{\log\delta}$, one has the measure of algebraic independence~\eqref{annexe_theo_ia2_result} at $\ul{x}=\left(1:y:\theta_1(y):\dots:\theta_n(y)\right)\in\mpp^{n+1}_{\mcc}$.
In particular,
\begin{equation*}
    \trdeg_{\mqq}\mqq\left(y,\theta_1(y),\dots,\theta_n(y)\right)\geq n+1-2\left\lfloor\frac{\log d}{\log\delta}\right\rfloor.
\end{equation*}
  \item Moreover, if $y\in\ol{\mqq}^*$, then 
there exists a constant $C>0$ such that for all projective variety $W\subset\mpp^n_{\mqq}$ of dimension $k < n+1-\frac{\log d}{\log\delta}$, one has the measure of algebraic independence~\eqref{annexe_theo_ia1_result} at $x=\left(1:\theta_1(y):\dots:\theta_n(y)\right)\in\mpp^n_{\mcc}$. In particular,
\begin{equation*}
    \trdeg_{\mqq}\mqq\left(\theta_1(y),\dots,\theta_n(y)\right)\geq n+1-\left\lfloor\frac{\log d}{\log\delta}\right\rfloor.
\end{equation*}
\end{enumerate}
\end{theorem}

If $p(z)\in\ol{\mqq}(z)$ then we can infer from our Theorem~\ref{annexe_theo_2} the following result,
improving on Corollary~6 from~\cite{ThTopfer1995}.

\begin{theorem} \label{annexe_theo_theta_2}  Let us consider different polynomials $q_1,\dots,q_n\in\ol{\mqq}[z]$ such that $1\leq\deg q_i<d-1$ and $|q_i(0)|>1$, $i=1,\dots,n$, and let $p$ be in $\ol{\mqq}(z)$. Define $\theta_i$, $i=1,\dots,n$, by~\eqref{annexe_def_theta_i}. Let $p_1,p_2\in\ol{\mqq}[z]$ be respectively numerator and denominator of $p$ (i.e. $p(z)=\frac{p_1(z)}{p_2(z)}$) and let
$$
\max(2,\deg p_2)<\deg p_1=d.
$$
Let $y\in\ol{\mqq}^*$ satisfies $\lim_{h\rightarrow\infty}p^{[h]}(y)=0$, $q_i\left(p^{[h]}(y)\right)\ne 0$ and $p^{[h]}(y)\ne 0$ for all $h\in\mnn$, $i=1,\dots,n$. Then, there exists a constant $C>0$ such that for every variety $W\subset\mpp^n_{\mqq}$  of dimension $k < n\left(2-\frac{\log d}{\log\delta}\right)$ the following holds true.
\begin{equation*} \index{Mesure d'independance algebrique@Mesure d'ind\'ependance alg\'ebrique}
     \log\Dist(\ul{x},W)\geq -Ch(W)^{1+\frac{\log d-\log\delta}{(2n-k)\log\delta-n\log d}n}\deg(W)^{\frac{2n}{n\left(2-\frac{\log d}{\log\delta}\right)-k}},
\end{equation*}
where $x=\left(1:\theta_1(y):\dots:\theta_n(y)\right)\in\mpp^n_{\mcc}$.
In particular,
\begin{equation*}
    \trdeg_{\mqq}\mqq\left(\theta_1(y),\dots,\theta_n(y)\right)\geq n\left(2-\frac{\log d}{\log\delta}\right).
\end{equation*}
\end{theorem}
\begin{remark}
In~\cite{Ch1980} G.V.Chudnovsky introduced the notion of "normality" of $n$-uplets $(x_1,\dots,x_n)\in\mcc^n$. According to his definition, the $n$-tuple  $(x_1,\dots,x_n)$ is \emph{normal} if it has a measure of algebraic independence of the form $\exp(-Ch(P)\psi(d(P)))$, i.e. if for all polynomial $P\in\ol{\mqq}[X_1,\dots,X_n]\setminus\{0\}$ one has the estimate
\begin{equation}\label{mesure_tau}
    |P(x_1,\dots,x_n)|\geq \exp(-Ch(P)\psi(d(P))),
\end{equation}
where $C>0$ is a real constant and $\psi:\mnn\rightarrow\mrr^{+}$ is an arbitrary function. If one has the estimate~(\ref{mesure_tau}) with $\psi(d)=d^\tau$ for some constant $\tau$ one says that this $n$-tuple has a measure of algebraic independence of Dirichlet's type. In this situation one also defines Dirichlet's exponent to be the infimum of  $\tau$ admitted for $\psi(d)=d^{\tau}$ in~(\ref{mesure_tau}). In~\cite{Ch1980} G.V.Chudnovsky mentioned that for $n\geq 2$ the explicit examples of normal $n$-tuples are quite rare, despite the fact that almost all (in the sense of Lebesgue measure) $n$-tuples of complex numbers are normal. 

Th.T\"opfer constructed a family of examples of normal $n$-tuples with Dirichlet's exponent $2n+2$ (see~\cite[Theorem~1 and Corollary~4]{ThTopfer1995}).

Our theorems 
improve the Dirichlet's exponent to $n+1$ for a large subfamily of these examples and allow also to produce new examples of normal $n$-tuples (due to the condition~(\ref{annexe_chi_ind_condition_b}) of Lemma~\ref{annexe_chi_ind} and also because our result are applicable to the collections of Mahler functions $f_1(z),\dots,f_n(z)$ which are not algebraically independent).
\end{remark}

\section{An application to Diophantine approximation to Mahler numbers} \label{section_U_numbers}

In case if $d=\delta$, our measures of algebraic independence given in Theorems~\ref{annexe_theo_ia1} and~\ref{annexe_theo_2} are optimal in $h(W)$. This allows us to prove new results on Diophantine approximations to a single Mahler number. Indeed, we are able to infer from Theorem~\ref{annexe_theo_ia1} that Mahler numbers are not $U$-numbers, so improving previous results~\cite{AB2011,AC2006,BBC2013}. 

\hidden{
The principal point in the proof of the main result of this section, Theorem~\ref{corollary_U_numbers}, is removal of Mahler's condition. For this purpose, we use the method presented in~\cite[Section~5]{BBC2013}. In our case, the removal of Mahler's condition gives quite a voluminous proof. Because of this, we begin with Theorem~\ref{theorem_U_numbers}, which provides the result under Mahler's condition, however its proof shows how the result on $U$-numbers follows from Theorem~\ref{annexe_theo_2} modulo Mahler's condition. Then, in the proof of Theorem~\ref{corollary_U_numbers} we focus on the removal of Mahler's condition.
}
The principal hurdle in the proof of the main result of this section, Theorem~\ref{corollary_U_numbers}, is removal of Mahler's condition. For this purpose, we use the method invented in~\cite[Section~5]{BBC2013}.

We begin with Theorem~\ref{theorem_U_numbers}, which provides the result under Mahler's condition, however its proof shows the deduction of the result on $U$-numbers from Theorem~\ref{annexe_theo_2} (modulo Mahler's condition). Then, in the proof of Theorem~\ref{corollary_U_numbers} we focus on the removal of Mahler's condition.

Let us denote by $I_{\ul{f}(\alpha)}\subset\mqq[X_1,\dots,X_n]$ the ideal of relations (with rational coefficients) between complex numbers $f_1(\alpha),\dots,f_n(\alpha)$. It will be convenient for us to homogenize the polynomials from $I_{\ul{f}(\alpha)}$ by using an additional variable $X_0$, hence defining a homogeneous ideal $I^h_{\ul{f}}\subset\mqq(z)[X_0,X_1,\dots,X_n]$. Consequently, we denote by $\cZ(I^h_{\ul{f}})\subset\mpp^n$ the set of common zeros of the homogeneous ideal $I^h_{\ul{f}}$.
\begin{theorem} \label{theorem_U_numbers}
Let $n\in\mnn$ and let $f_1(z),\dots,f_n(z)$ be a solution to the system~\eqref{systeme_1}, where $p(z)\in\ol{\mqq}(z)$ verifies $\deg p(z)=\ordz p(z)$. Let $\alpha\in\ol{\mqq}\setminus\{0\}$ verifies
\begin{equation*}
    p^{[h]}(\alpha)\rightarrow 0
\end{equation*}
as $h\rightarrow\infty$ and no iterate $p^{[h]}(\alpha)$, $h\in\mzz_{\geq 0}$, is a zero of $z\det A(z)$.

Then the number $f_1(\alpha)$ is not a $U$-number.

More precisely, one of the following two complimentary options holds true:
\begin{enumerate}

\item \label{theorem_U_numbers_option_1} $f_1(\alpha)\in\ol{\mqq}$, hence not in the class $U$.

\item \label{theorem_U_numbers_option_2} $f_1(\alpha)\not\in\ol{\mqq}$, and then there exist a constant $C_4>0$ such that for all 
$P\in\mzz[X]$
we have
$$
\left|P(f_1(\alpha))\right|\geq \exp\left(-C_4\left(h(P)+1\right)\deg(P)^{2n+1}\right).
$$
\end{enumerate}
\end{theorem}
\begin{proof}
The case~\ref{theorem_U_numbers_option_1} in the statement is trivial, so in the proof we will focus on the case~\ref{theorem_U_numbers_option_2}. Hence we assume in what follows
\begin{equation}~\label{hypothesis_not_algebraic}
f_1(\alpha)\not\in\ol{\mqq}.
\end{equation}

\hidden{
Let $t$ be defined by~\eqref{intro_trdeg}. It is well known that Mahler functions either belong to $\ol{\mqq}(z)$, either are transcendental over this field~\cite{Ni1996} (moreover, recently it has been proved that Mahler functions are \emph{hypertranscendental},  see~\cite{DHR2015}). So  assumption~\eqref{hypothesis_not_rational} implies that $f_1(z)\not\in\ol{\mqq(z)}$, hence $t\geq 1$. Moreover, we can assume, up to reindexing $f_i$, $i=2,\dots,n$, that $f_1(z),\dots,f_t(z)$ are algebraically independent.
}

Consider the ideal of the ring $\Q[X_1,\dots,X_n]$ defined by
$$
I_{\ul{f}(\alpha)}=\left\{ Q\in\Q[X_1,\dots,X_n] \mid Q(\ul{f}(\alpha)=0) \right\}.
$$
This ideal is prime, because it is the preimage of the zero ideal under the map
$$
\begin{aligned}
\Q[X_1,\dots,X_n]&\rightarrow\C,
\\
Q(X_1,\dots,X_n)&\to Q(f_1(\alpha),\dots,f_n(\alpha)).
\end{aligned}
$$
Furthermore, we have $\dim I_{\ul{f}(\alpha)}=t$, where $t$ is defined by~\eqref{intro_trdeg} (it follows, for example, from Theorem~\ref{annexe_theo_ia1}).

Let $P(X)\in\mzz[X]$. Consider the variety $W_{P}=\cZ(I_{\ul{f}(\alpha)}^h)\cap\cZ(P(X_1))\subset\mpp^n_{\mqq}$. As we consider only the values $f_1(\alpha)$ verifying~\eqref{hypothesis_not_algebraic}, we have $P\not\in I_{\ul{f}(\alpha)}$, by definition of $I_{\ul{f}(\alpha)}$. So necessarily $\dim W_P=t-1$
\hidden{
We claim that the dimension of $W_{P}$ is $t-1$,
\begin{equation} \label{theorem_U_numbers_dim_W_is_t-1}
\dim W_{P}=t-1.
\end{equation}

To see this, first note that the ideal $I_{\ul{f}}$ is prime: indeed, it is the preimage of the zero ideal under the map
$$
\begin{aligned}
\ol{\mqq}(z)[X_1,\dots,X_n]&\rightarrow\ol{\mqq}(z),
\\
P(X_1,\dots,X_n)&\to P(f_1(z),\dots,f_n(z)).
\end{aligned}
$$
Thus if $\dim W_{P}=t$, then necessarily the polynomial $P(X_1)$ belongs to $I_{\ul{f}}$, 
hence $P(f_1)=0$, but this contradicts assumption~\eqref{hypothesis_not_rational}. This contradiction proves~\eqref{theorem_U_numbers_dim_W_is_t-1}.
}
and then we can apply Theorem~\ref{annexe_theo_2} with $k=t-1$ and $d=\delta$. This theorem gives us the lower bound~\eqref{annexe_theo_2_result}, which specializes in our case to
\begin{equation} \label{theorem_U_numbers_lb_Dist}
\log\Dist(\ul{x},W_P)\geq -C_1h(W_P)\deg(W_P)^{2t},
\end{equation}
where $C_1>0$ is a constant.

Note that \cite[Proposition~4.11,~1), and~2)]{NP} implies
\begin{equation} \label{theorem_U_numbers_relations_Variety}
\begin{aligned}
\deg W_{P} &\leq \deg(I_{\ul{f}(\alpha)})\deg(P),
\\
h(W_{P}) &\leq h(P)\deg(I_{\ul{f}(\alpha)})+\deg(P)\left(h(I_{\ul{f}(\alpha)})+C'\deg(I_{\ul{f}})\right),
\end{aligned}
\end{equation}
where $C'>0$ is a constant (actually,~\cite[Proposition~4.11,~2]{NP} provides quite a simple explicit value for $C'$, but we don't need it here).

Furthermore, it follows from~\cite[Proposition~4.11,~3)]{NP} that
\begin{equation} \label{theorem_U_numbers_polynomial_estimate}
\log\left|P(f_1(\alpha))\right|\geq \log\Dist(\ul{x},W_P)-h(P)\deg(I_{\ul{f}(\alpha)})-\deg(P)\left(h(I_{\ul{f}(\alpha)})+C_2\deg(I_{\ul{f}})\right),
\end{equation}
where $C_2>0$ is one more constant.

Finally, we deduce
from~\eqref{theorem_U_numbers_lb_Dist}, \eqref{theorem_U_numbers_relations_Variety} and~\eqref{theorem_U_numbers_polynomial_estimate} that
\begin{equation} \label{theorem_U_number_final}
\log\left|P(f_1(\alpha))\right|\geq C_3 
\log\Dist(\ul{x},W_P)\geq -C_4\left(h(P)\deg(P)^{2t}+\deg(P)^{2t+1}\right),
\end{equation}
where $C_3, C_4>0$ are some constants.

This completes the proof.
\end{proof}

\hidden{
\begin{remark}
Note that in the 
classical framework of Mahler's function we can replace the assumption~\eqref{hypothesis_not_rational} by a weaker one,
$$
f_1(z)\not\in\ol{\mqq}(z).
$$
Indeed, it is well known that Mahler functions are either rational or transcendental. In fact, the rigidity in this case is even stronger, it has been proved that Mahler functions are \emph{hypertranscendental},  see~\cite{DHR2015}.

It seems that the methods used to prove this dichotomy for the classical Mahler functions should be extendibles to a more general framework of systems of the type~\eqref{systeme_1} considered in this paper. We plan to return to this topic in our future publications.
\end{remark}
}

\begin{theorem}
\label{corollary_U_numbers}
Let $n\in\mnn$ and let $f_1(z),\dots,f_n(z)$ be a solution to the system~\eqref{systeme_1}, where $p(z)\in\ol{\mqq}(z)$ verifies $\deg p(z)=\ordz p(z)$. Let $\alpha\in\ol{\mqq}\setminus\{0\}$ verifies
\begin{equation} \label{corollary_U_numbers_p_convergence}
p^{[h]}(\alpha)\rightarrow 0
\end{equation}
as $h\rightarrow\infty$. 

Then the number $f_1(\alpha)$ is not a $U$-number.

More precisely, one of the following two complimentary options holds true
\begin{enumerate}

\item $f_1(\alpha)\in\ol{\mqq}$, hence not in the class $U$.

\item $f_1(\alpha)\not\in\ol{\mqq}$, and then there exist constants $C_5,T>0$ such that for all non-zero 
$P\in\mzz[X]$
we have
\begin{equation} \label{corollary_U_numbers_conclusion}
\left|P(f_1(\alpha))\right|\geq \exp\left(-C_5\left(h(P)+1\right)\deg(P)^{T}\right).
\end{equation}
\end{enumerate}
\end{theorem}
\begin{proof}
The only thing we need to address additionally to the proof of Theorem~\ref{theorem_U_numbers} is the removal of Mahler's condition, that is the hypothesis that no iterate $p^{[h]}(\alpha)$, $h\in\mzz_{\geq 0}$, is a zero of $z\det A(z)$.

To remove Mahler's condition, note first that we can assume without loss of generality that $f_n(z)=1$ and that $\ul{f}(z)$ verifies
\begin{equation} \label{systeme_2}
\tilde{a}(z)\ul{f}(z)=\tilde{A}(z)\ul{f}(p(z)),
\end{equation}
where $\tilde{A}(z)$ is an $n\times n$ matrix with coefficients from $\ol{\mqq}[z]$, $\tilde{a}(z)\in\ol{\mqq}[z]$ and $\det\tilde{A}(z)$ is a non-zero polynomial. Indeed, if the function $1$ belongs to the linear span of $f_1(z),\dots,f_n(z)$ over $\ol{\mqq}(z)$ then~\eqref{systeme_2} is just another form of~\eqref{systeme_1}. Otherwise, if $1$ does not belong to the linear span of $f_1(z),\dots,f_n(z)$ over $\ol{\mqq}(z)$, then we can add function $1$ to the collection $\ul{f}$, increasing the dimension $n$ by 1.

Secondly, we claim that we can assume without loss of generality that
\begin{equation} \label{system_denominator_does_not_vanish}
\tilde{a}(p^{[k]}(\alpha))\ne 0,\quad k\in\mzz_{\geq 0}.
\end{equation}
To justify this claim, we use the method from~\cite[Section~5]{BBC2013}. Indeed, in any case we can factorize $\tilde{a}(z)=\gamma z^M \beta(z)$, where $\gamma\in\ol{\mqq}$, $M\in\mzz_{\geq 0}$, $\beta(z)\in\ol{\mqq}[z]$ and $\beta(0)=1$. Then, define
$$
\ul{g}(z):=\prod_{k=0}^{\infty}\beta(p^{[k]}(z))\ul{f}(z)
$$
(it is clear that the infinite product $\prod_{k=0}^{\infty}\beta(p^{[k]}(z))$ converges in $z$-adic valuation, and also it converges in the usual archimedean norm when specialized by $z=a$, where $a$ verifies the condition~\eqref{corollary_U_numbers_p_convergence}).

Because of~\eqref{systeme_2}, we have the following functional system for $\ul{g}(z)$:
\begin{equation} \label{systeme_3}
\gamma z^M \ul{g}(z)=\tilde{A}(z)\ul{g}(p(z)).
\end{equation}
Moreover, differentiating~\eqref{systeme_3} we find
$$
\gamma z^M \begin{pmatrix}
\ul{g}(z)\\
\ul{g}'(z)
\end{pmatrix}
=
\begin{pmatrix}
\tilde{A}(z) & 0_{n\times n}\\
* & p'(z) \tilde{A}(z)
\end{pmatrix}
\begin{pmatrix}
\ul{g}(p(z))\\
\ul{g}'(p(z))
\end{pmatrix}.
$$
More generally, introducing the notation, for $s\in\mnn$,
$$
\ul{g}_s(z)=\begin{pmatrix}
\ul{g}(z)\\
\vdots\\
\ul{g}^{(s)}(z)
\end{pmatrix},
$$
we have
\begin{equation} \label{systeme_35}
\gamma z^M \ul{g}_s(z)
=
\tilde{B}(z)
\ul{g}_s(z^d),
\end{equation}
where
$$
\tilde{B}(z)=\begin{pmatrix}
\tilde{A}(z)& 0 & \dots & 0 \\
* & p'(z) \tilde{A}(z) & & \vdots\\
\vdots & \ddots & \ddots & \vdots \\
* & \cdots & * & \left(p'(z)\right)^s \tilde{A}(z)
\end{pmatrix}.
$$

As $\beta(z)$ is a non-zero polynomial, there exists a real $u>0$ such that $\beta(z)\ne 0$ for all $0<|z|<u$. Then, there exists $N,s\in\mzz_{\geq 0}$ and $T(z)\in\ol{\mqq}$, depending on $\tilde{a}(z)$ and $\alpha\in\ol{\mqq}$, $0<|\alpha|<1$ only, such that $\beta(p^{[k]}(\alpha))\ne 0$ for all $k\geq N+1$, $T(\alpha)\ne 0$ and
$$
\prod_{k=0}^{N}\beta(p^{[k]})=\left(z-\alpha\right)^sT(z).
$$
If $s=0$, then we readily have~\eqref{system_denominator_does_not_vanish}. Otherwise, by L'H\^opital's rule we have
$$
\ul{f}(\alpha)=\frac{\ul{g}^{(s)}(\alpha)}{s!T(\alpha)\prod_{k\geq N+1}\beta(p^{[k]}(\alpha))}.
$$
Next, define functions $L_{i,j}(z)$, $1\leq i\leq n$, $0\leq j\leq n$, by
$$
L_{i,j}(z):=\frac{g_i^{(j)}(z)}{s!T(\alpha)\prod_{k\geq N+1}\beta(p^{[k]}(\alpha))}
$$
and denote by $\ul{L}(z)$ the vector function with coordinates $1\leq i\leq n,0\leq j\leq s$.

The system~\eqref{systeme_35} implies
\begin{equation} \label{systeme_4}
\gamma z^L \beta(p^{[N]}(z)) \ul{L}(z)
=
\tilde{B}(z)
\ul{L}(p(z)).
\end{equation}
Now, the system~\eqref{systeme_4} is of the same shape as~\eqref{systeme_2}. Moreover, $\gamma z^L \beta(p^{[N]}(z))$ does not vanish on $p^{[k]}(\alpha)$ for any $k\in\mnn$, by the choice of $N$. Finally, we have
\begin{equation} \label{f_eq_L}
f_1(\alpha)=L_{1,s}(\alpha),
\end{equation}
so to prove our 
theorem
we can consider system~\eqref{systeme_4} in place of system~\eqref{systeme_2}. Up to changing notations, considering system~\eqref{systeme_4}  boils down to considering~\eqref{systeme_2} with the additional assumption~\eqref{system_denominator_does_not_vanish}. 

Furthermore, the theorem trivially holds true if $f_1(\alpha)\in\ol{\mqq}$, that is in the case~1 of the statement of the theorem. So, we need only to treat the case~2 of the statement of the theorem, hence assuming additionally
\begin{equation} \label{f_not_in_algebraic}
f_1(\alpha)\not\in\ol{\mqq}.
\end{equation}
Then, the combination of~\eqref{f_eq_L} and~\eqref{f_not_in_algebraic} implies
$$
L_{1,s}(\alpha)\not\in\ol{\mqq},
$$
and then necessarily
\begin{equation} \label{L_1s_not_in_alg_Qz}
L_{1,s}\not\in\ol{\mqq(z)}.
\end{equation}

The net outcome of our considerations above in this proof is that it is enough to prove our theorem for the system~\eqref{systeme_2} with additional assumption~\eqref{system_denominator_does_not_vanish}, which we assume until the end of this proof.

Further, note that $\det\tilde{A}(z)$ is a non-zero polynomial, hence there exists a real $v>0$, depending on $\det\tilde{A}(z)$ only (hence on system~\eqref{systeme_1} only), such that $\det\tilde{A}(z)$ does not vanish for all $0<|z|<v$.

Let $K\in\N$ be the minimal integer such that $\left|p^{[K]}\right|\in (0,u)$. The functional system~\eqref{systeme_2} implies
\begin{equation} \label{corollary_U_numbers_split}
\prod_{k=0}^{K-1}\tilde{a}(p^{[k]}(z))\ul{f}(z)=\prod_{k=0}^{K-1}\tilde{A}(p^{[k]}(z))\ul{f}(p^{[K]}(z)).
\end{equation}
In particular, for some $c_1(z),\dots,c_n(z)\in\ol{\mqq}[z]$ we have
$$
\prod_{k=0}^{K-1}\tilde{a}(p^{[k]}(z))f_1(z)=\sum_{k=1}^nc_k(z)f_k(p^{[K]}(z)).
$$
So, if $f_1(\alpha)\not\in\ol{\mqq}$, then necessarily
$$
\sum_{k=1}^n\frac{c_k(\alpha)}{\prod_{k=0}^{K-1}\tilde{a}(p^{[k]}(\alpha))}f_k(p^{[K]}(z))
$$
is a non-zero function, moreover transcendental over $\mqq(z)$. Thus, for any polynomial $P(X)\in\ol{\mqq}[X]$, the polynomial
$$
P\left(\sum_{k=1}^n\frac{c_k(\alpha)}{\prod_{k=0}^{K-1}\tilde{a}(p^{[k]}(\alpha))}X_k\right)
$$
does not belong to the ideal $I_{\ul{f}}$. So, similarly to the proof of Theorem~\ref{theorem_U_numbers}, we deduce that the variety $\cZ\left(P\left(\sum_{k=1}^nc_k(\alpha)X_k\right)\right)\cap\cZ\left(I_{\ul{f}}\right)$ has dimension $t-1$. The rest of the proof is the same as the proof of Theorem~\ref{theorem_U_numbers}.

So, we apply Theorem~\ref{annexe_theo_2} 
and then use~\eqref{theorem_U_numbers_relations_Variety} to find, similarly to~\eqref{theorem_U_number_final}, that for any $y\in\ol{\mqq}$ we have
$$
\left|P\left(\sum_{k=1}^n\frac{c_k(\alpha)}{\prod_{k=0}^{K-1}\tilde{a}(p^{[k]}(\alpha))}f_k(y)\right)\right|\geq\exp\left(-C_3\left(h(P)\deg(P)^{2\tilde{t}}+\deg(P)^{2\tilde{t}+1}\right)\right),
$$
where $\tilde{t}$ is equal to the degree of transcendence, over $\ol{\mqq}(z)$, of the set of functions $L_{i,j}(z)$, $1\leq i\leq n$, $0\leq j\leq s$, introduced above. Finally, we substitute $y=p^{[K]}(\alpha)$ and use~\eqref{corollary_U_numbers_split} to get the conclusion~\eqref{corollary_U_numbers_conclusion}.
\end{proof}
\begin{remark}
Our lower bound is better, asymptotically in $\deg(P)$, than the one given in~\cite[Theorem~8.1]{BBC2013}. As for the constant $C_3$, in principle its value can be tracked through this our article and the article~\cite{EZ2013_2} (the latter article gives the constant for Theorem~\ref{theoNishioka}, which influences the constant $C$ in~\eqref{annexe_theo_2_result}, hence the constant $C_3$). It has to be said that this tracking is a very laborious work and the final value of constant $C_3$ will be huge, in particular, far much bigger than the multiplicative constants given in~\cite{BBC2013}. 
\end{remark}
\begin{remark}
In earlier versions of this text, the author erroneously claimed Theorems very similar to Theorems~\ref{theorem_U_numbers} and~\ref{corollary_U_numbers} from the present text, but with the case~1 replaced by a more restrictive version $f_1(z)\in\ol{\Q}(z)$ (such claims, if they have had been true, would have given a slightly sharper dichotomy for Mahler's functions). The intended proofs were very similar to the ones from the present version of the text, with some evident adjustments: for example, the ideal of relations between functions $I_{\ul{f}}$ was used in place of $I_{\ul{f}(\alpha)}$ etc. However, that line of reasoning contained a flaw related to the fact that for an idel $I$ defined over $\ol{\Q}(z)$ the specialization given by $z\to\alpha\in\ol{\Q}$ may lead to an ideal defined over $\ol{\Q}$ of another dimension.

The author is indebted to Boris Adamczewski for bringing his attention to the point that there exist Mahler functions with algebraic coefficients taking algebraic values at an algebraic point $\alpha$, even in the case if $\alpha$ verifies Mahler's condition. So dichotomies between cases~1 and~2 in  Theorems~\ref{theorem_U_numbers} and~\ref{corollary_U_numbers} in fact can not be sharpened by replacing the case~1 by $f_1(z)\in\ol{\Q}(z)$.

The following example is taken from~\cite[Example~25]{PP_2011}.

Let $a\geq 2$ be an integer and let $f(z)=\sum_{k=0}^{\infty}z^{a^k}$. Define functions
\begin{eqnarray*}
f_1(z)&:=&(z-\frac12)f(z),\\
f_2(z)&:=&z f(z).
\end{eqnarray*}
The functions $1$, $f_1(z)$ and $f_2(z)$ form a solution to a Mahler type system of functional equations 
\begin{equation*}
\begin{pmatrix}
1\\f_1(z)\\f_2(z)
\end{pmatrix}
=\begin{pmatrix}
1 & 0 & 0\\
-z\cdot (z^a-1/2) & 1 & z^{a-1}-1\\
-z^{a+1} & 0 & z^{a-1}
\end{pmatrix}^{-1}
\begin{pmatrix}
1\\f_1(z^a)\\f_2(z^a)
\end{pmatrix}.
\end{equation*}
One can directly check that the point $z=1/2$ verifies Mahler's condition for this functional system. At the same time, $f_1\left(\frac12\right)=0$ while $f_1(z)$ is a transcendental function (it is evident, for example, from the fact that the unit circle is the natural boundary for this series). So, $f_1(z)\not\in\ol{\Q}(z)$, however its value at the point $z=1/2$, $f_1\left(\frac12\right)=0$, is of course algebraic, hence no general estimate of the type given in the case~2 of Theorem~\ref{theorem_U_numbers} is possible.
\end{remark}

\section{Criterion for algebraic independence.} \label{section_cai}


In this section we elaborate a criterion for algebraic independence, Theorem~\ref{CplM}, adapted to our situation. We deduce it from a general result~\cite[Theorem~5.1]{J1996}, see Theorem~\ref{critereJadot} below.

In what follows, we use the notions of degree and height of a projective variety as well as the distance from a point of a projective space to its subvariety. These notions are defined in~\cite[Chapters~5 and~6]{NP}.

\smallskip

\begin{theorem} \label{critereJadot} (Particular case of~\cite[Theorem~5.1]{J1996}) Let $K$ be a number field, 
$k$ an integer from $[0,n]$ and $\ul{\theta}=(\theta_0,\dots,\theta_n)\in\mcc^{n+1}$.

Let $\mu\geq 0$, $\sigma,\delta\geq 1$, $\tau\geq 0$ and $U$  be real numbers.

Assume that the following objects exist.

\begin{itemize}
  \item A strictly increasing sequence of real numbers $(S_i)_{0\leq i\leq l}$ satisfying
  \begin{equation} \label{critereJadot_restriction_1}
    0<S_0\leq\tau+\log 2 \text{ and } S_{l-1}<U\leq S_{l},
  \end{equation}
  \item for $1\leq i\leq l$, a homogeneous polynomial $Q_i\in K[X_0,\dots,X_n]$ such that
  \begin{enumerate}
    \item $\deg Q_i=\delta$,
    \item $h(\omega_{\delta}^*(Q_i))\leq\tau$, 
    where 
    $\omega_{\delta}^*(Q_i)$ denotes the polynomial $$\sum_{|\ul{\alpha}|=d}\binom{d}{\ul{\alpha}}^{-\frac12}a_{|\ul{\alpha}|}\ul{X}^{|\ul{\alpha}|}$$ for $Q_i=\sum_{|\ul{\alpha}|=d}a_{|\ul{\alpha}|}\ul{X}^{|\ul{\alpha}|}$,
    \item $\frac{|Q_i(\ul{\theta})|}{|\ul{\theta}|^{\deg Q_i}}\leq\ee^{-S_i}$, 
    \item the polynomial $Q_i$ has no zeros in the ball $B\left(\ul{\theta},\ee^{-(S_{i-1}+\mu)\sigma}\right)$ of $\mpp_n(\mcc)$ with centre at $\ul{\theta}$ and of radius $\ee^{-(S_{i-1}+\mu)\sigma}$. 
  \end{enumerate}
\end{itemize}

Let $\mathcal{I}\subset K[X_0,\dots,X_n]$ be a homogeneous ideal of dimension $k$.

Define $t_{\tau,\ul{d}}\eqdef h_{\ul{d}}(\mathcal{I})+(\tau+(k+1)\delta\log(n+1))\deg_{\ul{d}}(\mathcal{I})$ o\`u $\ul{d}=(\delta,\dots,\delta)\in\mnn^{k+1}$
 and assume that the following condition is realized
  \begin{equation}\label{critereJadot_codition_1}
    \sigma^{k+1}\left(\frac{[K:\mqq]}{n_{\infty}}t_{\tau,\ul{d}}(\mathcal{I})+\left(\mu+\log 2+\frac{\log\delta}{2\delta^k}\right)\deg_{\ul{d}}\mathcal{I}\right) \leq U,
  \end{equation}
  where $n_{\infty}$ denotes the index of ramification of the valuation $|\cdot|_{\infty}$.
  Then,
  \begin{equation*}
    \log\Dist\left(\ul{\theta},\V(\mathcal{I})\right)\geq-U.
  \end{equation*}
\end{theorem}
\begin{proof}
This is Theorem~5.1 of~\cite{J1996} with $\kappa=\tau$ and the condition $B$ applied to the absolute archimedean value.
\end{proof}

Corollary~\ref{corTheseJadot59} below is a straightforward generalization of Corollary~5.9 from~\cite{J1996} to the case of an arbitrary number field $K$.
\begin{corollary}\label{corTheseJadot59} (see Corollary~5.9 in~\cite{J1996}) Let $K$ be a number field,
$k$ an integer from the range to $[0,n]$ and $\ul{\theta}=(1,\theta_1,\dots,\theta_n)\in\mcc^{n+1}$.

Let $\delta$, $\tau$, $\sigma$ and $U$ be real numbers satisfying $\sigma,\delta\geq 1$ and $U>\tau\geq 3(k+1)\delta\log(n+1)$.

Assume that for every real $s$ verifying $\tau<s\leq U$, there exists a polynomial $R_s$ from $K[X_1,\dots,X_n]$ such that
\begin{itemize}
  \item $\deg R_s \leq \delta$,
  \item $h\left(\omega_{\deg R_s}^*(R_s)\right)\leq\tau$,
  \item $\ee^{-\sigma s+2\tau}\leq\frac{|R_s(\ul{\theta})|}{\left(1+|\ul{\theta}|^2\right)^{\frac{\deg R_s}{2}}}\leq\ee^{-s}$
\end{itemize}
 Then for every homogeneous ideal $\mathcal{I}\subset K[X_0,\dots,X_n]$ of dimension $k$, of degree $D$ and of height $H$ satisfying
 \begin{equation} \label{corTheseJadot59_hyp_ie_main}
    2\delta^k[K:\mqq]\left(\delta H+\tau D\right)\leq\frac{U}{\sigma^{k+1}}
 \end{equation}
 we have
 \begin{equation*}
    \log\Dist\left(\ul{\theta},\V(\mathcal{I})\right)\geq-U.
 \end{equation*}
\end{corollary}

\begin{proof}
We use the line of reasoning from the proof of Corollary~5.9 in~\cite{J1996}.

Set $\theta_0=1$, and let $t$ be an index such that $\theta_t=\max\left(\theta_0,\dots,\theta_n\right)$ and denote $\ul{\theta}'=\left(\frac{\theta_0}{\theta_t},\dots,\frac{\theta_n}{\theta_t}\right)$.

We are going to use the following preliminary result:

{ \it Let $Q\in K[X_0,\dots,X_n]$ such that $Q(\ul{X})=\sum_{|\ul{\alpha}|=d}a_{\ul{\alpha}}\ul{X}^{\ul{\alpha}}$ be a homogeneous polynomial.
If this polynomial has at least one zero in the ball centred at  $\ul{\theta}'=\left(\frac{\theta_0}{\theta_t},\dots,\frac{\theta_n}{\theta_t}\right)$ and of radius
$R<\frac{1}{\sqrt{n+1}}$ for the euclidean distance of the affine chart $X_t\ne 0$ then
\begin{equation*}
    \frac{|Q(\ul{\theta}')|}{|\ul{\theta}'|^{2\delta}}\leq R 2^{\delta-1}\sqrt{n+1}\delta\sqrt{\sum_{|\alpha|=\delta}\frac{|a_{\ul{\alpha}}|^2}{\binom{\delta}{\ul{\alpha}}}}.
\end{equation*}
}

The proof of this result can be found in~\cite{J1996}, page~127.

\begin{notation}
 Let $P$ be a polynomial from $K[X_1,\dots,X_n]$. We denote by $^hP$ its homogenization, that is $^hP$ is a homogeneous polynomial from $K[X_0,X_1,\dots,X_n]$ defined by
$$
^hP(X_0,\dots,X_n):=X_0^{\deg P}P\left(\frac{X_1}{X_0},\dots,\frac{X_n}{X_0}\right),
$$
where $\deg P$ denotes the total degree of $P$ with respect to $X_1,\dots,X_n$.
\end{notation}

We verify, using Corollary~\ref{corTheseJadot59}, that the polynomial $^hR_s$ has no zeros in the ball centred at $(1,\theta_1,\dots,\theta_n)$ and of the radius $\ee^{-\sigma s}$. Note that the condition $\tau\geq 3(k+1)\delta\log(n+1)$ implies $\tau>(\delta-1)\log 2+\frac12\log(n+1)+\log\delta$. Also we have $\sqrt{\sum_{|\alpha|=\delta}\frac{|a_{\ul{\alpha}}|^2}{\binom{\delta}{\ul{\alpha}}}}\leq\ee^{\tau}$, we infer from the condition $\ee^{-\sigma s+2\tau}\leq\frac{|R_s(\ul{\theta})|}{\left(1+|\ul{\theta}|^2\right)^{\frac{\deg R_s}{2}}}$ the following inequality:
\begin{equation*}
    \ee^{-\sigma s}2^{\delta-1}\sqrt{n+1}\delta\sqrt{\sum_{|\alpha|=\delta}\frac{|a_{\ul{\alpha}}|^2}{\binom{\delta}{\ul{\alpha}}}}
    <\frac{|^hR_s(\ul{\theta}')|}{|\ul{\theta}'|^{\delta}}.
\end{equation*}
This inequality and our preliminary statement allow us to conclude that the polynomial $^hR_s$ has no zeros
in a ball of centre $(1,\theta_0,\dots,\theta_n)$ and of radius $\ee^{-\sigma s}$.


Let $l=\left[\frac{U-S_0}{\mu}\right]$, where $\mu=2(k+1)\delta\log(n+1)-\log 2-\frac{\log\delta}{2\delta^k}>0$. Consider the following sequence:
\begin{equation*}
    \left\{\begin{aligned}
        &0<S_0=\tau+\log 2,\\
        &S_i=S_{i-1}+\mu, \qquad i=1,\dots,l-1,\\
        &S_l=U.
    \end{aligned}\right.
\end{equation*}
We readily verify that for every $i$, $1\leq i\leq l$, the polynomials $Q_i=^h\!\!R_{S_i}$ verify
\begin{itemize}
  \item $\deg(Q_i)=\delta$,
  \item $h(\omega_{\delta}^*(Q_i))\leq\tau$,
  \item $\frac{|Q_i(\ul{\theta}')|}{|\ul{\theta}'|^{\delta}}\leq\ee^{-S_i}$,
  \item The polynomial $Q_i$ has no zeros in the ball centred at $(1,\theta_0,\dots,\theta_n)$ and of radius $\ee^{-\sigma(S_{i-1}+\mu)}$.
\end{itemize}

Moreover, we have for $\ul{d}=(\delta,\dots,\delta)\in\mnn^{k+1}$
\begin{equation*}
\begin{aligned}[]
    \frac{[K:\mqq]}{n_{\infty}}t_{\tau,\ul{d}}(\mathcal{I})+&\left(\mu+\log 2+\frac{\log\delta}{2\delta^k}\right)\deg_{\ul{d}}(\mathcal{I})\\
    &\leq\delta^k[K:\mqq]\left(\delta H+\left(\tau+3(k+1)\delta\log(n+1)\right)D\right)\\
    &\leq 2\delta^k[K:\mqq]\left(\delta H+\tau D\right).
\end{aligned}
\end{equation*}
Assumption~\eqref{corTheseJadot59_hyp_ie_main} implies
\begin{equation*}
    \frac{[K:\mqq]}{n_{\infty}}t_{\tau,\ul{d}}(\mathcal{I})+\left(\mu+\log 2+\frac{\log\delta}{2\delta^k}\right)\deg_{\ul{d}}(\mathcal{I})\leq\frac{U}{\sigma^{k+1}}
\end{equation*}
and hence the condition~\eqref{critereJadot_codition_1} of Theorem~\ref{critereJadot} holds true.

Thus all the hypothesis of Theoreme~\ref{critereJadot} are verified, so we can apply this theorem with the parameters $\delta$, $\tau$, $\sigma$, $\mu$, $U$ as above. We find with this theorem
\begin{equation*}
    \log\Dist\left(\ul{\theta},\V(\mathcal{I})\right)\geq-U.
\end{equation*}
and this gives us the conclusion of Corollary~\ref{corTheseJadot59}.
\end{proof}

Theorem~\ref{CplM} below is our principal tool in proofs of algebraic independence and establishing measures of algebraic independence. It is essentially Criterion for the measures from~\cite{PP_KF} with one small technical adjustment: in the statement of Theorem~\ref{CplM} we allow polynomials $R_s$ to have coefficients in a number field $K$, and not only in $\mzz$ as in~\cite{PP_KF}. The proof from~\cite{PP_KF} still perfectly works in this, more general, case. I reproduce this proof below for the commodity of the reader, following~\cite{PP_KF}. The only change needed in the proof is that one has to use Corollary~\ref{corTheseJadot59} instead of~\cite[Corollary~5.9]{J1996} (cited in~\cite{PP_KF} as "Theorem from page~5.").

This generalisation allows to apply approximation polynomials $R_s$ with coefficients in a number field $K$, and not only in $\mzz$. This (easy) improvement is important for our purposes. At the same time, I don't see how to deduce Theorem~\ref{CplM} directly from the statement of Criterion for the measures in~\cite{PP_KF}, without a reference to its proof.
\hidden{
Theorem~\ref{CplM} below is our principal tool in proofs of algebraic independence and establishing measures of algebraic independence. It is proved in~\cite{PP_KF} with $\mzz$ in place of $K$. The proof given there still works if we replace $\mzz$ with an arbitrary number field $K$, the only change needed is that one has to use Corollary~\ref{corTheseJadot59} instead of Corollary~5.9 in~\cite{J1996} (cited in~\cite{PP_KF} as "Theorem from page~5.").  For the commodity of the reader, I reproduce below the proof from~\cite{PP_KF} with this minor modification. This generalisation, allowing auxilliary polynomials $R_s$ to have coefficients in a number field $K$ and not only in $\mzz$, is important for our purposes in this paper. At the same time, I don't see how to deduce Theorem~\ref{CplM} directly from Criterion for the measures in~\cite{PP_KF}, without reference to its proof.
}
\begin{theorem} \label{CplM} (Criterion for the measures,~\cite[page~5]{PP_KF}) Let $m,k,\lambda,\delta_1,\dots,\delta_m\in\mnn$, $\sigma,\tau,U\in\mrr_{>0}$ and $\ul{x}=(x_1,\dots,x_m)\in\mcc^m$ be such that $0\leq k<m$, $\lambda\geq 1$, $\delta_1\geq\dots\geq\delta_m\geq 1$, $U>\tau\geq 4(k+1)\log\left(\delta_1\cdots\delta_m(1+m^2)\right)$. Let $K$ be a number field. Assume that for every real $\tau\lambda\leq s\leq U$ there exists a polynomial $R_s\in K[X_1,\dots,X_m]$ of degree $<e(s)\delta_i$ in $X_i$, of length $\leq \exp(e(s)\tau)$ and satisfying
\begin{equation}\label{CplM_approximation}
    \exp\left(-s\sigma+2e(s)\tau\right)\leq\frac{|R_s(\ul{x})|}{\prod_{i=1}^m(1+|x_i|^2)^{\delta_i/2}}\leq\exp\left(-s-e(s)(\delta_1+\dots+\delta_m)\right)
\end{equation}
where $e(s):=1+\max_{1\leq i\leq m}\left[\frac{\deg_{X_i}R_s}{\delta_i}\right]\leq\lambda$. Then, for every algebraic variety $V\subset\mathbb{A}^m(\mcc)$ defined over $\mqq$, of dimension $k$, and satisfying
\begin{equation}\label{CplM_VarietyRestriction}
    [K:\mqq]\cdot 3\lambda^{k+1}\delta_1\dots\delta_k\left(\delta_{k+1}t(V)+\tau\deg(V)\right)\leq U/(\sigma m)^{k+1},
\end{equation}
one has
\begin{equation*}
    \log\Dist(x',V')\geq-U,
\end{equation*}
where $x'=(1:x_1:\dots:x_m)\in\mpp^m_{\mcc}$ and $V'$ denotes the completion of $V$ in $\mpp^m_{\mcc}$.

\end{theorem}

\begin{proof} Consider the completion $W$ of $V$ in $\left(\mpp^1\right)^m$ and the point $\tilde{x}=\left(1:x_1,\dots,1:x_m\right)\in\left(\mpp^1\right)^m$. Let $\phi:\left(\mpp^1\right)^m\hookrightarrow\mpp^N$, where $N=(\delta_1+1)\cdots(\delta_m+1)-1$ be the embedding defined by
\begin{equation*}
    \phi(x_{0,1}:x_{1,1},\dots,x_{0,m}:x_{1,m})=\left(\dots:\prod_{i=1}^mx_{0,i}^{\alpha_i}x_{1,i}^{\delta_i-\alpha_i}:\dots\right)_{\ul{\alpha}\in\mnn^m,0\leq\alpha_i\leq\delta_i}.
\end{equation*}
We readily verify, as it is done in~\cite{PP_AH}, III, Proposition~1 and~\cite{J1996}, \S2.3~(b), Lemma~2.13,
\begin{eqnarray*}
  d\left(\phi(W)\right) &=& k!\sum_{\substack{i\in\{0,1\}^m\\|i|=k}}d_i(W)\delta_1^{i_1}\cdots\delta_m^{i_m},\\
  h\left(\phi(W)\right) &=& (k+1)!\sum_{\substack{i\in\{0,1\}^m\\|i|=k+1}}\left(h_i(W)+\sum_{\substack{\alpha=1\\i_{\alpha}\ne 0}}d_{\hat{i_{\alpha}}}(W)\right)\delta_1^{i_1}\cdots\delta_m^{i_m},
\end{eqnarray*}
where $d_i(W)$ and $h_i(W)$ denote the multihomogeneous degrees and heights of $W$, and where $\hat{i_{\alpha}}$ is obtained from $i$ by setting the $\alpha$-th component equal to 0 (so that $|\hat{i_{\alpha}}|=k$). In particular, $k!\sum d_i(W)$ and $(k+1)!\sum h_i(W)$ are the degree and the height of $V$ embedded in $\mpp^{2^m-1}$ with Segre's embedding and are equal respectively to $m^kd(W)$ and $m^{k+1}h(W)$. Every polynomial $R_s$, suitably homogenized, can be represented as an inverse image by $\phi$ of a form $P_s\in K[X_0,\dots,X_N]$ of degree $e(s)\leq\lambda$ and of the length $\leq\ee^{e(s)\tau}\leq\ee^{\tau\lambda}$. Moreover, if $y=\phi(\tilde{x})$ we have
\begin{equation*}
    1\leq\frac{\prod_{i=1}^n(1+|x_i|^2)^{\delta_i/2}}{|y|}\leq\ee^{\delta_1+\dots+\delta_m},
\end{equation*}
where $|y|=\sum_{i=0}^m|y_i|$.

Hence
\begin{equation*} 
    \exp\left(-s\sigma+2e(s)\tau\right)\leq\frac{|P_s(y)|}{|y|^{\deg P_s}}\leq\exp\left(-s\right).
\end{equation*}
At the same time, the condition~(\ref{CplM_VarietyRestriction}) implies
\begin{equation*}
    2\delta^{k+1}\left(h\left(\phi(W)\right)+(k+1)\tau d\left(\phi(W)\right)\right)\leq U'/\sigma^{k+1}
\end{equation*}
where $U'=U-m^{k+1}\delta_1\cdots\delta_k\log\left(\delta_1(1+m^2)\right)d(V)$. Indeed, as $\delta_1\geq\cdots\geq\delta_m$, we have
\begin{eqnarray*}
  d\left(\phi(W)\right) &=& m^kd(V)\delta_1\cdots\delta_k, \\
  h\left(\phi(W)\right) &=& m^{k+1}t(V)\delta_1\cdots\delta_{k+1}.
\end{eqnarray*}
So we can apply Corollary~\ref{corTheseJadot59} to $\phi(W)$ and this completes the proof, because by Proposition~3.9 of~\cite{J1996} we can verify
\begin{equation*}
    \begin{aligned}
        \log\Dist(x',V') &\geq \log\Dist(\tilde{x},W)-\frac{k+1}{1}\cdot\log\left(\delta_1\cdots\delta_m(1+m^2)\right)\cdot d\left(\phi(W)\right)\\
                       &\geq \log\Dist(y,\phi(W))-m^{k+1}\delta_1\cdots\delta_k\log\left(\delta_1(1+m^2)\right)d\left(\phi(W)\right)\\
                       &\geq \log -U'-m^{k+1}\delta_1\cdots\delta_k\log\left(\delta_1(1+m^2)\right)d\left(\phi(W)\right)=-U,
    \end{aligned}
\end{equation*}
and the claim of the theorem readily follows.
\end{proof}

\hidden{
\begin{remark}
Theorem~\ref{CplM} is proved in~\cite{PP_KF} with $\mzz$ in place of $K$. The proof given there still works if we replace $\mzz$ with an arbitrary number field, the only change needed is that one has to use Corollary~\ref{corTheseJadot59} instead of Corollary~5.9 in~\cite{J1996} (cited in~\cite{PP_KF} as "Theorem from page~5."). Here above, we used the line of the proof from~\cite{PP_KF}. 
\end{remark}
}

\section{Extrapolative construction} \label{section_K_functions}

In this section we present (a simple version of) the extrapolative construction elaborated in~\cite{PP_KF}. A version of the construction that we need is stated in Proposition~\ref{annexe_P7}. As the statement of Proposition~\ref{annexe_P7} uses in a significant way the notion of $K$-functions, introduced in~\cite{PP_KF}, we remind in this section the relavant definitions from~\cite{PP_KF}. Please note that here we consider only a restricted version which is sufficient for our purposes. We refer the reader to~\cite{PP_KF} for a more general version.

We start with some notations. Let $m\geq 1$ be an integer and let $f_1(z),\dots,f_m(z)$ be functions analytic at $0$ with a radius on convergence at least 1. For every integer $i\in\mnn$ and every formal power series (for instance, for every function analytic at $z=0$) $g\in\mcc((\b{z}))$ we denote by $D_i(g)$ the $i$-th coefficient of $g$ (we follow the notation introduced in~\cite{PP_KF}). Further, for every $\Dc\subset\mnn^n$  and every $i\in\mnn$ we define
\begin{equation} \label{def_D_Dc}
    \ul{D}_i f^{\Dc}(0)=\left(D_{i'}(f_1^{\alpha_1}\cdots f_m^{\alpha_m})(0)\mid\alpha\in \Dc, i'=0,\dots,i\right),
\end{equation}
this is a vector from $\ol{\mqq}^{(i+1)\card \Dc}$. For every increasing function $\psi:\mnn\rightarrow\mnn^*$ and a finite set of cardinality $\card\Dc\geq 2\psi(0)$, we also denote $i_{\Dc}=i_{\Dc,\psi}$ the biggest integer such that $2i_{\Dc}\psi(i_{\Dc}-1)\leq\card\Dc$.

We also use the notation
\begin{equation*}
    |\Dc|:=\max_{\alpha\in\Dc}\sum_{j=1}^{n+1}|\alpha_j|.
\end{equation*}

\smallskip

\begin{definition}
Let $n\in\mnn^*$, let $A$ be an infinite set of subsets of $\mnn^n$, and let functions
\begin{equation*}
    \psi:\mnn\rightarrow\mnn^*,\quad\phi:A\times\mnn\rightarrow\mrr_{\geq \ee},
\end{equation*}
where $\mrr_{\geq \ee}$ denotes the set of real numbers greater or equal than $e$, be such that for $\Dc\in A$ as above the functions $\psi(i)$ and $\phi_{\Dc}(i)$ are increasing in $i$, the function $\frac{\log\phi_{\Dc}(i)}{i}$ is decreasing in $i$ for $i\geq i_{\Dc}$ and $\liminf_{|\Dc|\rightarrow\infty}\frac{\log\phi_{\Dc}(i_{\Dc})}{i_{\Dc}}=0$. Moreover, assume $\phi_{\Dc}(i)\geq 4\sqrt{2}(i+1)\psi(i)$ for all $i\geq i_{\Dc}$.

We say that a family $(f_1,\dots f_m)$ of functions analytic in  $B(0,1)$ forms a \emph{system of $K$-functions of type $(\psi,\phi)$} if for all $i\in\mnn$ and for all $\Dc\in A$ we have
\begin{equation*}
    \ul{D}_i\ul{f}(0)\subset\ol{\mqq},\quad\left[\mqq(\ul{D}_i\ul{f}(0)):\mqq\right]\leq\psi(i),\quad h\left(\ul{D}_i\ul{f}^{\Dc}(0)\right)\leq\phi_\Dc(i).
\end{equation*}
\end{definition}

\begin{definition}
Let $c\geq 1$ be a real number. We say that $\Dc\in A$ is \emph{$c$-admissible} for a system of $K$-functions $f_1,\dots,f_{n+1}$ if for all $Q\in\mzz[X_1,...,X_{n+1}]\setminus\{0\}$, supported by $\Dc$ (that is of the form $\sum_{\alpha\in \Dc}q_{\alpha}X_1^{\alpha_1}\cdots X_{n+1}^{\alpha_{n+1}}$), of length $\leq\sqrt{2}\cdot\card\Dc\cdot\phi_{\Dc}(i_{\Dc})$, we have
\begin{equation*}
    \ordz Q(f_1,\dots,f_{n+1})\leq c\cdot\card\Dc.
\end{equation*}
\end{definition}

\begin{remark} \label{rem_full_admissible} It follows from the definitions that for any integers $n, t, D\geq 1$, $t\leq n$, a set
\begin{equation}\label{defDcanonique}
    \Dc=\{\ul{h}\in\mnn^{n+1}\mid h_0<D_0, \; h_i<D,\;i=1,\dots,t, \; h_i=0,\; i=t+1,\dots,n\}
\end{equation}
is $c$-admissible for a system of $K$-functions $z, \b{f}_1,\dots,\b{f}_{n}$ if and only if $\ull{f}=(z, \b{f}_1,\dots,\b{f}_{t})$ satisfies multiplicity lemma with the optimal exponent and the multiplicative constant $c$, that is if for any non-zero polynomial $P\in\mcc[z,X_1,\dots,X_t]$ we have
$$
\ordz P(z,f_1(z),\dots,f_t(z))\leq c D_0 D^t.
$$
\end{remark}
Theorem~\ref{theoNishioka} below shows that Remark~\ref{rem_full_admissible} is applicable in the case of Mahler's functions~\eqref{systeme_1}, at least if $D_0\geq D$.
\begin{theorem} \label{theoNishioka}
Let  $(f_1(\b{z}),\dots,f_n(\b{z}))$ be an $n$-tuple of functions $\mcc\rightarrow\mcc$ analytic at $z=0$ that form a solution to the system of functional equations~\eqref{systeme_1}. Also, assume~\eqref{intro_trdeg} and assume that $f_1(z),\dots,f_t(z)$ are algebraically independent.
Then there exists a constant $K_1$ such that for any polynomial $P\in\AnneauDePolynomes$ 
either $P(\ull{f})=0$ or
\begin{equation} \label{intro_theoNishioka_conclusion}
    \ordz(P(\ull{f})) \leq K_1(\deg_{\ul{X'}}P + \deg_{\ul{X}}P + 1)(\deg_{\ul{X}} P + 1)^t.
\end{equation}
\end{theorem}
\begin{proof}
See~\cite{EZ2013_2}, Theorem~5.8.
\end{proof}

The following proposition is a particular case of the general construction developed in~\cite{PP_KF}. We use this result in Section~\ref{section_PS} to construct polynomials with nice approximation properties.
\begin{proposition} \label{annexe_P7} (A particular case of Proposition~7 in~\cite{PP_KF}).
Let $K$ be a number field and let $c'\geq 1$ and let $f_1,...,f_{n+1}$ be a system of $K$-functions of type $(\psi,\phi)$. Then,
\begin{enumerate}
  \item for all $\Dc\in A$ there exists a polynomial $P\in\mzz[X_1,\dots,X_{n+1}]\setminus\{0\}$, supported by $\Dc$, of length
$$\leq\sqrt{2}\cdot\card \Dc\cdot\phi_{\Dc}(i_{\Dc})$$
and such that $T_0:=\ordz F(\b{z})\geq i_{\Dc}$, where
\begin{equation} \label{annexe_P7_def_F}
F=P(f_1,...,f_{n+1});
\end{equation}

  \item \label{annexe_P7_part_two} under assumption $\frac{\log\phi_{\Dc}(i_{\Dc})}{i_{\Dc}}\leq\frac{1}{24}$, for all real numbers $r'$, $r''$ such that $0<r''\leq r'<r^4$, where
  \begin{equation} \label{P7_def_r}
    r:=1-\frac{12\log\phi_{\Dc}(T_0)}{T_0},
  \end{equation}
  for every positive integer $0<N\leq T_0/8$ and for every point 
$$
z\in\ol{B(0,r')}\setminus B(0,r'')
$$
satisfying
  \begin{equation}\label{P7_condition1}
    \frac{-N\cdot\log \left(\frac{r'}{r}\right)-N\cdot\log\left(\frac{1+|z|/r'}{1+|z|\cdot r'/r^2}\right)      }{\log\phi_{\Dc}(T_0)}\geq 19
  \end{equation}
  there exists a positive integer $0<i\leq N$ such that
  \begin{equation*}
    \left(\frac{r''}{4}\right)^{(c'+4)T_0}\leq |D_iF(z)|\leq\left(\frac{r'}{(1-\sqrt{r'})^2}\right)^{T_0/16}.
  \end{equation*}
  Moreover, if $\Dc$ is $c$-admissible for $f_1,\dots,f_{n+1}$, then $T_0\leq c\cdot\card\Dc$.
\end{enumerate}
\end{proposition}

\section{Polynomial sequences} \label{section_PS}

In this section we construct sequences of polynomials with nice approximation properties. Our main result in this section is Proposition~\ref{existence_P_DT}, which is used in proofs of Theorems~\ref{annexe_theo_ia1} and~\ref{annexe_theo_ia2}. We derive this result from general extrapolative construction, Proposition~\ref{annexe_P7}. The proof of Theorem~\ref{annexe_theo_2} makes appeal to a slightly different Proposition~\ref{existence_P_DT_rational_p}.

We start with an auxilliary lemma, which is Lemma~2 from~\cite{ThTopfer1995}. We provide a proof (following the lines of the proof in~\cite{ThTopfer1995}) 
for the commodity of the reader and, simultaneously, to fix a minor issue with the proof given in~\cite{ThTopfer1995}.
\begin{lemma} \label{lemma_db_p_T}
Let $p(z)=p_1(z)/p_2(z)$ be a rational function with $\delta=\ordz p(z)\geq 2$ and let $y\in\mcc$ satisfies $p^{[T]}(y)\ne 0$ for all $T\in\mnn$. Assume
\begin{equation} \label{lemma_db_p_T_convergence}
\lim_{T\rightarrow\infty}p^{[T]}(y)=0.
\end{equation}
Then there exist constants $0<c_3',c_3''<1$ and $T_s>0$, depending on $p$ and $y$ only, such that
\begin{equation} \label{lemma_db_p_T_result}
    |c_3'|^{\delta^T}\leq|p^{[T]}(y)|\leq |c_3''|^{\delta^T},
\end{equation}
for all $T\geq T_s$.
\end{lemma}
\begin{proof}
As $p$ is a rational fraction and $z=0$ is a zero of order $\delta\geq 2$, we have that
\begin{equation} \label{lemma_db_p_T_presentation_p}
p(z)=z^{\delta}g(z),
\end{equation}
where $g$ is a rational fraction satisfying $g(0)\ne 0$. In particular, $g$ is a continuous function defined in a neighbourhood $U$ of 0. So, for a sufficiently small $U$, there exist constants $\tilde{c}_3',\tilde{c}_3''>0$ such that
\begin{equation} \label{lemma_db_p_T_db_g}
\tilde{c}_3'\leq |g(z)|\leq\tilde{c}_3''
\end{equation}
for every $z\in U$.

Combining~\eqref{lemma_db_p_T_presentation_p} and~\eqref{lemma_db_p_T_db_g}, we find
\begin{equation} \label{lemma_db_p_T_first_bound}
\tilde{c}_3'|z|^{\delta}\leq |p(z)|\leq \tilde{c}_3''|z|^{\delta}
\end{equation}

By assumption~\eqref{lemma_db_p_T_convergence}, there is an index $\hat{T}\in\mnn$ such that for every $T\geq\hat{T}$ one has $p^{[T]}(y)\in U$. So iterating~\eqref{lemma_db_p_T_first_bound} we find~\eqref{lemma_db_p_T_result} with $c_3',c_3''>0$. The upper bound $c_3',c_3''<1$ follows from just established~\eqref{lemma_db_p_T_result} and assumption~\eqref{lemma_db_p_T_convergence}.
\end{proof}

The following lemma gives an upper bound for the coefficients of the solution of~\eqref{systeme_1} in the special case when $p(z)$ is a polynomial with algebraic coefficients. This lemma  is proven in~\cite{ThTopfer1995}, in Lemma~\ref{ub_coefficients_poly} below we translate this statement to our notation.
\begin{lemma}[Lemma~12 in~\cite{ThTopfer1995}] \label{ub_coefficients_poly}
Let functions $f_1,...,f_{n+1}$ satisfy~\eqref{systeme_1} with $p(z)\in\ol{\mqq}[z]$. Then, all the numbers $\ul{D}_i\ull{f}(0)$, $i\in\mnn$, belong to a fixed number field and enjoy the following upper bound for the height:
 \begin{equation} \label{ub_hDtf}
    \exp\left(h(\ul{D}_i\ul{f}^{\Dc}(0))\right)\leq (i+2)^{c_2tD}.
 \end{equation}
\end{lemma}
\begin{proof}
Lemma~12 of~\cite{ThTopfer1995} shows that the numbers $\ul{D}_i\ull{f}(0)$ belong to a fixed number field and provides the 
claimed
upper bound for heights.

Note that one has the following relations between the notations of~\cite{ThTopfer1995} and our notations introduced in~\eqref{def_D_Dc} and in the text preceding this equality:
 \begin{equation*}
    \ul{D}_i\ull{f}(0)=\left(f_{1,i},\dots,f_{n+1,i}\right)
 \end{equation*}
 and
  \begin{equation*}
    \ul{D}_i\ull{f}^{\Dc}(0)=\left(f_{i}^{(\ul{j})}\mid\ul{j}\in\Dc\right),
 \end{equation*}
and the height in~\eqref{ub_hDtf} is the product of the denominator and the house of these vectors.
\end{proof}

From this moment on, we denote by $K$ the number field generated by coefficients of Taylor expansions of $f_1,\dots, f_{n+1}$ at the origin together with coefficients of $p(z)$, $a(z)$ and of all polynomials involved in the matrices $A(z)$ and $B(z)$ of system~\eqref{systeme_1} (in the proofs of Theorem~\ref{annexe_theo_ia1} and Theorem~\ref{annexe_theo_2}, when we consider values of $f_1(y),\dots,f_{n+1}(y)$ at an algebraic point $y$, we extend $K$ by $y\in\ol{\mqq}$ as well).

\begin{remark}
The upper bound~\eqref{ub_hDtf} implies that the series $f_k(\b{z})$, $k=1,\dots,n+1$, converge in the circle $|\b{z}|<1$.
\end{remark}
The next lemma embodies an important step in the proof of Proposition~\ref{existence_P_DT} below, which, in its turn, plays an important role in proofs of our main results, Theorems~\ref{annexe_theo_ia1} and~\ref{annexe_theo_ia2}. We isolate this step in Lemma~\ref{lemma_P7} below in order to make the reading easier.
\begin{lemma} \label{lemma_P7}
Assume the situation of Proposition~\ref{annexe_P7}, that is let $f_1,...,f_{n+1}$ be a system of $K$-functions of type $(\psi,\phi)$, and assume moreover that $\phi$ is given by
$$
\phi_{\Dc}(i)=(i+2)^{c_2tD},
$$
where $c_2$ is a constant independent of $i$ and $\Dc$.

Let the function $F$ be defined by~\eqref{annexe_P7_def_F}, for Lemma~\ref{annexe_P7} applied with $y\in\mcc$, $0<|y|<1$, $S=\{(0,y)\}$, $r'=r''=|y|$, $D_0\geq D\geq 1$ and $\Dc$ given by~\eqref{defDcanonique}. Assume moreover that functions $f_1,...,f_{n+1}$ satisfy~\eqref{systeme_1} and define
\begin{equation} \label{lemma_P7_def_q}
q:=\det A.
\end{equation}
Then for every integer $T\geq 0$ there exists a polynomial $P_{D_0,D,T}\in K[X_0,X_1,\dots,X_n]$ with coefficients from $\ol{\mqq}$, of degree in $\b{z}$ upper bounded by
\begin{equation} \label{ub_P_DT_deg_z}
\deg_z P_{D_0,D,T} \leq c_4D_0(d^T+\dots+1)\leq c_5 D_0 d^T,
\end{equation}
of degree in $\ul{X}$ upper bounded by
\begin{equation} \label{ub_P_DT_deg_X}
\deg_{\ul{X}}P_{D_0,D,T}<n D,
\end{equation}
of the length not exceeding
\begin{equation} \label{ub_P_DT_L}
L(P_{D_0,D,T})\leq\exp(c_6D_0(d^T+\log D_0)),
\end{equation}
and such that
\begin{equation} \label{annexe_defPT}
    \left(\prod_{i=0}^{T-1}q(p^{[i]}(y))\right)^{nD}\cdot F\left(p^{[T]}(y)\right)=P_{D_0,D,T}\left(y,f_1(y),\dots,f_n(y)\right).
\end{equation}
\end{lemma}
\begin{proof}
For $T=0$, the existence of $P_{D_0,D,0}$ readily follows from the definition of $F$, see~\eqref{annexe_P7_def_F} , where $\Dc$ is given by~\eqref{defDcanonique}.

For $T\geq 1$ we proceed with recurrence. Assume we established the existence of $P_{D_0,D,T}$ verifying~\eqref{ub_P_DT_deg_z}-\eqref{annexe_defPT}, for some indices $D_0,D,T$, and we want to prove the existence of $P_{D_0,D,T+1}$ verifying~\eqref{ub_P_DT_deg_z}-\eqref{annexe_defPT} as well. To this end, substitute $p(y)$ in place of $y$ to the equality~\eqref{annexe_defPT} for $P_{D_0,D,T}$ and apply to the right hand side the equality
\begin{equation} \label{annexe_P7_fpz_expression}
\ull{f}(p(\b{y}))=A(\b{y})^{-1}\left(a(\b{y})\ull{f}(\b{y})-B(\b{y})\right),
\end{equation}
which readily follows from~\eqref{systeme_1}. As the result, we infer 
the equality
\begin{equation} \label{lemma_P7_Q}
\left(\prod_{i=1}^{T}q(p^{[i]}(y))\right)^{nD}\cdot F\left(p^{[T+1]}(y)\right)=Q\left(y,f_1(y),\dots,f_n(y)\right),
\end{equation}
where $Q=Q(y,X_1,\dots,X_n)$ is a polynomial in $\ul{X}$ of the same degree as $P_{D_0,D,T}$, because~\eqref{annexe_P7_fpz_expression} is linear in $\ul{f}(y)$. So we have
$$
\deg_{\ul{X}} Q=\deg_{\ul{X}}P_{D_0,D,T}<n D.
$$
At the same time, $Q$ is a rational fraction in $y$, and the common denominator of its terms is a polynomial divisor of $\left(\det A\right)^{\deg_{\ul{X}}P_{D_0,D,T}}$ (because the denominator in the right hand side of~\eqref{annexe_P7_fpz_expression} is a polynomial in $y$ dividing $\det A$). Hence the denominator of the rational fraction $Q$ is a polynomial from $\mqq[y]$ dividing $\left(\det A\right)^{nD}=q^{nD}$ (we recall the notation~\eqref{lemma_P7_def_q}). So multiplying both sides of~\eqref{lemma_P7_Q} by $q^{nD}$, we find the equality~\eqref{annexe_defPT} for $T+1$. 

So, we define
\begin{multline} \label{annexe_recurrence_Pn}
    P_{D_0,D,T+1}(y,\ul{f}(y))\\=q(y)^{nD}P_{D_0,D,T}\left(p(y),A^{-1}(y)a(y)\ul{f}(y)-A^{-1}B(y)\right).
\end{multline}
The arguments above prove that with this definition the equality~\eqref{annexe_defPT} holds true for $P_{D_0,D,T+1}(y,\ul{f}(y))$. It remains us to verify  the upper bounds~\eqref{ub_P_DT_deg_z} and~\eqref{ub_P_DT_L}, that is upper bounds for the degree in $y$ and for the length of $P_{D_0,D,T+1}:=q^{nD}Q$.

By the recurrence hypothesis 
we find, with the notation $A_0(y)=q(y)a(y)A^{-1}(y)$ and $B_0(y)=q(y)A^{-1}B(y)$ and assuming the lower bound $c_4\geq\max\left(\deg_{\b{z}}A_0,\deg_{\b{z}}B_0\right)+n\deg(q)$,
\begin{equation*}
\begin{aligned}
    \deg_{\b{z}}P_{D_0,D,T+1}&\leq nD\deg q + c_4dD_0\left(d^{T}+\dots+1\right)+D_0\max\left(\deg_{\b{z}}A_0,\deg_{\b{z}}B_0\right)\\
                     &\leq c_4D_0\left(d^{T+1}+\dots+1\right).
\end{aligned}
\end{equation*}
So for $c_4=n\cdot\deg q+\max\left(\deg_{\b{z}}A_0,\deg_{\b{z}}B_0\right)$ the upper bound~\eqref{ub_P_DT_deg_z} holds true by recurrence. We readily find
\begin{equation} \label{annexe_mesure_dia_2}
    \deg_{\b{z}}P_{D_0,D,T+1} \leq c_4D_0\frac{d^{T+2}-1}{d-1}, 
\end{equation}
hence
\begin{equation} \label{annexe_mesure_dia_3}
    \deg_{\b{z}}P_{D_0,D,T+1} \leq c_4 D_0 d^{T+2}.
\end{equation}

We proceed with the proof of the upper bound for the length $L(P_{D_0,D,T+1})$. We deduce from~(\ref{annexe_recurrence_Pn}) the upper bound
\begin{equation*}
\begin{aligned}
    L(P_{D_0,D,T+1})&\leq L(P_{D_0,D,T})\cdot L(q)^{nD}\cdot L(p)^{\deg_{\b{z}}P_{D_0,D,T}}\\&\qquad\times\left(L(A_0(y))+L(B_0(y))\right)^D\\
          &\leq L(P_{D_0,D,T})\exp\left(c_6'(D+\deg_{\b{z}}P_{D_0,D,T})\right),
\end{aligned}
\end{equation*}
and taking into account~(\ref{annexe_mesure_dia_3}) we find
\begin{equation*}
    L(P_{D_0,D,T+1})\leq L(P_{D_0,D,T})\exp\left(c_6D_0d^{T}\right).
\end{equation*}
Applying the hypothesis of the recurrence we infer
\begin{equation} \label{annexe_theo_ia1_majL}
\begin{aligned}
    L(P_{D_0,D,T+1})&\leq\exp\left(c_6D_0(d^{T}+\dots+1+\log D_0)\right)
\\
&\leq \exp\left(c_6D_0(d^{T+1}+\log D_0)\right).
\end{aligned}
\end{equation}
We conclude with recurrence that~\eqref{ub_P_DT_L} holds true and
it completes the proof of the lemma.
\end{proof}
\begin{proposition} \label{existence_P_DT}
Let functions $f_1,...,f_{n}$ satisfy~\eqref{systeme_1} with $p(z)\in\ol{\mqq}[z]$, and let $y\in\mcc$ satisfies $p^{[T]}(y)\ne 0$ for all $T\in\mnn$. Assume that for  an integer $1\leq t\leq n$ we have~\eqref{intro_trdeg} and $f_1(z),\dots,f_t(z)$ are algebraically independent over $\mcc(z)$. Assume
\begin{equation}  \label{existence_P_DT_lim}
\lim_{T\rightarrow\infty}p^{[T]}(y)=0.
\end{equation}
Then for any $D_0\geq D\in\mnn$ and $T\in\mnn$ big enough there exists a polynomial $P_{D_0,D,T}(z,X_1,\dots,X_n)$ satisfying~\eqref{ub_P_DT_deg_z}, \eqref{ub_P_DT_deg_X} and \eqref{ub_P_DT_L}.
If moreover
\begin{equation} \label{annexe_defT}
    T\geq T_1(D_0, D),
\end{equation}
where
\begin{equation} \label{def_T_zero_D}
T_1(D_0, D)\eqdef\frac{\log D+\log\log D_0+\log\left(25 t(t+1) c_2\right)-\log|\log c_3''|}{\log \delta},
\end{equation}
and $c_3''$ is as in~\eqref{lemma_db_p_T_result}, the constant $c_2$ is defined by~\eqref{ub_hDtf} and $K_1$ is defined by~\eqref{intro_theoNishioka_conclusion}, then this polynomial satisfies as well
\begin{equation} \label{Q_double_estimation_statement}
    \exp(-c_7D_0D^{t}\delta^T)\leq|P_{D_0,D,T}(\ul{x})|\leq\exp(-c_8D_0D^{t}\delta^T),
\end{equation}
where
\begin{equation*}
    \ul{x}=\left(y,f_1(y),\dots,f_n(y)\right)\in\mcc^{n+1}
\end{equation*}
and $c_7, c_8\in\mrr^+$ are positive constants.
\end{proposition}
\begin{proof}
Let $\Dc\subset\mnn^{n+1}$ be defined by~\eqref{defDcanonique}.
\hidden{
Lemma~12 of~\cite{ThTopfer1995} shows that the numbers $\ul{D}_i\ull{f}(0)$ belong to a fixed number field and provides the following upper bound for heights:
 \begin{equation} \label{ub_hDtf}
    h(\ul{D}_i\ul{f}^{\Dc}(0))\leq (i+2)^{c_2tD}.
 \end{equation}
Note that one has the following relations between the notations of~\cite{ThTopfer1995} and our notations introduced in~\eqref{def_D_Dc} and in the text preceding this equality:
 \begin{equation*}
    \ul{D}_i\ull{f}(0)=\left(f_{1,i},\dots,f_{n,i}\right)
 \end{equation*}
 and
  \begin{equation*}
    \ul{D}_i\ull{f}^{\Dc}(0)=\left(f_{i}^{(\ul{j})}\mid\ul{j}\in\Dc\right),
 \end{equation*}
and the height in~\eqref{ub_hDtf} is the product of the denominator and the house of these vectors. In particular, the upper bound~\eqref{ub_hDtf} implies that the series $f_k(\b{z})$, $k=1,\dots,n+1$, converge in the circle $|\b{z}|<1$.
}

By Lemma~\ref{ub_coefficients_poly}, $\b{z}, f_1(\b{z}),\dots,f_{n}(\b{z})$ is a system of $K$-functions of type $(\psi,\phi)$ with
\begin{equation} \label{existence_P_DT_def_psi}
\psi(i)=c_1
\end{equation}
and
\begin{equation} \label{existence_P_DT_def_phi}
\phi_{\Dc}(i)=
(i+2)^{c_2tD}.
\end{equation}
Note that the exponent in the right hand side of~\eqref{existence_P_DT_def_phi} does not contain $D_0$, this is because multiplication by any power $z^N$ does not increase the size of Taylor coefficients of $\ul{f}^{\ul{\alpha}}(z)$ at $z=0$.

Also, note that~\eqref{existence_P_DT_def_psi} implies $i_{\cD}=\lfloor \card\cD/2c_1 \rfloor=\lfloor D_0D^t/2c_1 \rfloor$. 

Assumption~\eqref{existence_P_DT_lim} allows us to apply Lemma~\ref{lemma_db_p_T},  so finding, for $T$ big enough, the double bound~\eqref{lemma_db_p_T_result}. Note that in the estimates~\eqref{lemma_db_p_T_result} the constants $0<c_3',c_3''<1$ depend on $p$ and $y$ only.



Further, apply Lemma~\ref{lemma_P7} to get a sequence of polynomials $P_{D_0,D,T}$ enjoying upper bounds \eqref{ub_P_DT_deg_z}, \eqref{ub_P_DT_deg_X} and~\eqref{ub_P_DT_L} 
and verifying as well~\eqref{annexe_defPT}.
Also, recall the notation $T_0=\ordz F(\b{z})$, where $F(z)$ is defined by~\eqref{annexe_P7_def_F}, and recall that by Proposition~\ref{annexe_P7},
\begin{equation} \label{T0_geq_iD}
T_0\geq i_{\cD}=\lfloor \card\cD/2c_1 \rfloor, 
\end{equation}


It remains us to show that under assumption~\eqref{annexe_defT} the polynomial $P_{D_0,D,T}$ verifies~\eqref{Q_double_estimation_statement}. To this end, we apply part~\ref{annexe_P7_part_two} of Proposition~\ref{annexe_P7}, with the set $\Dc$ defined by~(\ref{defDcanonique}), 
$r'=r''=|p^{[T]}(y)|$, $N=1$ and $z=p^{[T]}(y)$. Condition~\eqref{P7_condition1} in this case is equivalent to
\begin{equation} \label{P7_condition1_equivalent_one}
 \frac{\log\left(\frac{r^2+\left|p^{[T]}(y)\right|^2}{2\cdot r\cdot\left|p^{[T]}(y)\right|}\right)      }{\log\phi_{\Dc}(T_0)}\geq 19,
\end{equation}
where $r$ is defined by~\eqref{P7_def_r}.

Now we are going to prove~\eqref{P7_condition1_equivalent_one}. First, note that
\begin{equation} \label{P7_LDM}
T_0\leq 2K_1 D_0D^{t}
\end{equation}
by Theorem~\ref{theoNishioka} 
(see also  Remark~\ref{rem_full_admissible} and recall our notation $T_0=\ordz F(\b{z})$). Hence, for $D$ large enough, we have
\begin{equation} \label{P_DT_ub}
19\log\phi_{\Dc}(T_0)\leq 19\log\left((T_0+2)^{c_2tD}\right)\leq 20c_2tD\left(t\log D+\log D_0\right).
\end{equation}
At the same time, definition of $r$, \eqref{P7_def_r}, implies that for $D$ large enough $r$ is as close to 1 as we need. So, taking into account~\eqref{lemma_db_p_T_result}, we find
\begin{equation} \label{P_DT_lb}
\log\left(\frac{r^2+\left|p^{[T]}(y)\right|^2}{2\cdot r\cdot\left|p^{[T]}(y)\right|}\right)\geq \frac45\delta^T\left|\log c_3''\right|,
\end{equation}
where the constant $c_3''$ is of course the same as in~\eqref{lemma_db_p_T_result}.

Now, the inequality~\eqref{P7_condition1_equivalent_one} follows by comparing~\eqref{P_DT_ub} and~\eqref{P_DT_lb} and than by using~\eqref{def_T_zero_D}.


Thus we can apply part~\ref{annexe_P7_part_two} of Theorem~\ref{annexe_P7} (taking into account as well the double bound~\eqref{lemma_db_p_T_result}) to find that
\begin{equation} \label{Q_double_estimation}
    \exp(-\tilde{c}_7D_0D^{t}\delta^T)\leq|F(p^{[T]}(y))|\leq\exp(-\tilde{c}_8D_0D^{t}\delta^T),
\end{equation}
where $\tilde{c}_7,\tilde{c}_8\in\mrr^+$ are two positive constants.

In view of~\eqref{annexe_defPT}, the polynomial $P_{D_0,D,T}$ satisfies
\begin{equation} \label{P_DT_equality}
P_{D_0,D,T}(\ul{x})=\left(\prod_{i=0}^{T-1}q(p^{[i]}(y))\right)^{nD}\cdot F(p^{[T]}(y)).
\end{equation}
At the same time, the double bound~\eqref{lemma_db_p_T_result} implies, for some constants $\tilde{c}_9,\tilde{c}_{10}>0$,
\begin{equation} \label{P_DT_auxilliary_db}
\exp(-\tilde{c}_9D_0D^{t}\delta^T)\leq\prod_{i=0}^{T-1}q(p^{[i]}(y))\leq \tilde{c}_{10}^{nD}.
\end{equation}
Combining~\eqref{Q_double_estimation}, \eqref{P_DT_equality} and~\eqref{P_DT_auxilliary_db} we find~\eqref{Q_double_estimation_statement}
and it completes the proof of the proposition.
\end{proof}

In the case when $p(z)$ is a rational function, and not a polynomial, we have only a far weaker upper bound for the height of coefficients in Taylor expansion at the origin of $f_i(z)$, $i=1,\dots,n$. Consequently, Proposition~\ref{existence_P_DT_rational_p} below, which is a counterpart of Proposition~\ref{existence_P_DT} in the case when $p(z)$ is a rational function and not just a polynomial, gives polynomials with weaker bounds for their length. Otherwise, proof of Proposition~\ref{existence_P_DT_rational_p} is analogous to the proof of Proposition~\ref{existence_P_DT}. the only thing modified is the upper bound for the coefficients of $f_i(z)$, $i=1,\dots,n$, which is given by the following lemma.
\begin{lemma}[Lemma~7 in~\cite{ThTopfer1995}] \label{ub_coefficients_rational}
Let functions $f_1,...,f_{n}$ satisfy~\eqref{systeme_1} with $p(z)\in\ol{\mqq}[z]$. Then, all the numbers $\ul{D}_i\ull{f}(0)$, $i\in\mnn$, belong to a fixed number field and enjoy the following upper bound for the height:
 \begin{equation} \label{ub_hDtf}
    \exp\left(h(\ul{D}_i\ul{f}^{\Dc}(0))\right)\leq \exp\left(c_{14}(i+tD)\right).
 \end{equation}
\end{lemma}
\begin{proof}
This is Lemma~7 in~\cite{ThTopfer1995}. The comments on the notation made in the proof of Lemma~\ref{ub_coefficients_poly} are applicable here as well.
\end{proof}

\begin{proposition} \label{existence_P_DT_rational_p}
Let functions $f_1,...,f_{n}$ satisfy~\eqref{systeme_1} with $p(z)\in\ol{\mqq}(z)$, and let $y\in\mcc$ satisfies $p^{[T]}(y)\ne 0$ for all $T\in\mnn$. Assume that for  an integer $1\leq t\leq n+1$ we have~\eqref{intro_trdeg} and $f_1(z),\dots,f_t(z)$ are algebraically independent over $\mcc(z)$. Assume
\begin{equation}  \label{existence_P_DT_lim_rational}
\lim_{T\rightarrow\infty}p^{[T]}(y)=0.
\end{equation}
Then for any $D_0\geq D\in\mnn$ and $T\in\mnn$ big enough there exists a polynomial $R_{D_0,D,T}(z,X_1,\dots,X_n)$ satisfying~\eqref{ub_P_DT_deg_z}, \eqref{ub_P_DT_deg_X} and \eqref{ub_P_DT_L}.
\begin{equation} \label{ub_P_DT_deg_z_rational}
\deg_z R_{D_0,D,T} \leq c_4D_0(d^T+\dots+1)\leq c_5 D_0 d^T,
\end{equation}
of degree in $\ul{X}$ upper bounded by
\begin{equation} \label{ub_P_DT_deg_X_rational}
\deg_{\ul{X}}R_{D_0,D,T}<n D,
\end{equation}
of the length not exceeding
\begin{equation} \label{ub_P_DT_L_rational}
L(R_{D_0,D,T})\leq\exp((D_0+1)(D'^t+d^T)),
\end{equation}

If moreover
\begin{equation} \label{annexe_defT_rational}
    T\geq T_2(D_0, D),
\end{equation}
where
\begin{equation} \label{def_T_zero_D_rational}
T_2(D_0, D)\eqdef\frac{t\log D+\log D_0+\log\left|\frac{5c_{15}}{4\log c_3''}\right|}{\log \delta},
\end{equation}
and $c_3''$ is as in~\eqref{lemma_db_p_T_result}, the constant $c_2$ is defined by~\eqref{ub_hDtf} and $K_1$ is defined by~\eqref{intro_theoNishioka_conclusion}, then this polynomial satisfies as well~\eqref{Q_double_estimation_statement}.
\hidden{
\begin{equation} \label{Q_double_estimation_statement_rational}
    \exp(-c_7D_0D^{t}\delta^T)\leq|R_{D_0,D,T}(\ul{x})|\leq\exp(-c_8D_0D^{t}\delta^T),
\end{equation}
where
\begin{equation*}
    \ul{x}=\left(y,f_1(y),\dots,f_n(y)\right)\in\mcc^{n+1}
\end{equation*}
and $c_7, c_8\in\mrr^+$ are positive constants.
}
\end{proposition}
\begin{proof}
The proof is very similar to the proof of Proposition~\ref{existence_P_DT}, the only difference is the use of Lemma~\ref{ub_coefficients_rational} in place of Lemma~\ref{ub_coefficients_poly}. Because of this, we only present the main outline of the proof here, referring the reader to the proof of Proposition~\ref{existence_P_DT} for more explanations.

So, let $\Dc\subset\mnn^{n+1}$ be defined by~\eqref{defDcanonique}.
By Lemma~\ref{ub_coefficients_rational}, $\b{z}, f_1(\b{z}),\dots,f_{n}(\b{z})$ is a system of $K$-functions of type $(\psi,\phi)$ with
\begin{equation} \label{existence_P_DT_def_psi}
\psi(i)=c_1
\end{equation}
and the function $\phi_{\Dc}(i)$ given by the right hand side of~\eqref{ub_hDtf}.

Assumption~\eqref{existence_P_DT_lim_rational} allows us to apply Lemma~\ref{lemma_db_p_T},  so finding, for $T$ big enough, the double bound~\eqref{lemma_db_p_T_result}. Note that in the estimates~\eqref{lemma_db_p_T_result} the constants $0<c_3',c_3''<1$ depend on $p$ and $y$ only.



Further, using the same procedure as described in the proof of Lemma~\ref{lemma_P7} we get a sequence of polynomials $R_{D_0,D,T}$ enjoying upper bounds \eqref{ub_P_DT_deg_z_rational}, \eqref{ub_P_DT_deg_X_rational} and~\eqref{ub_P_DT_L_rational} 
and verifying as well~\eqref{annexe_defPT}.


It remains us to show that under assumption~\eqref{annexe_defT_rational} the polynomial $P_{D_0,D,T}$ verifies~\eqref{Q_double_estimation_statement}. To this end, we apply part~\ref{annexe_P7_part_two} of Proposition~\ref{annexe_P7}, with the set $\Dc$ defined by~(\ref{defDcanonique}), 
$r'=r''=|p^{[T]}(y)|$, $N=1$ and $z=p^{[T]}(y)$. Similarly to the proof of Proposition~\ref{existence_P_DT}, condition~\eqref{P7_condition1} is equivalent to~\eqref{P7_condition1_equivalent_one}.

Now we are going to verify~\eqref{P7_condition1_equivalent_one}. First, note that by Theorem~\ref{theoNishioka} we have
\begin{equation} \label{P7_LDM_rational}
T_0\leq 2K_1 D_0D^{t}.
\end{equation}
Hence, for $D$ large enough, we have
\begin{equation} \label{P_DT_ub_rational}
19\log\phi_{\Dc}(T_0)\leq 19c_{14}\left(T_0+tD\right)\leq c_{15} D_0D'^t.
\end{equation}
At the same time, we have the lower bound~\eqref{P_DT_lb} by using the double bound~\eqref{lemma_db_p_T_result} in exactly the same way as in the proof of Proposition~\ref{existence_P_DT}. Then, the inequality~\eqref{P7_condition1_equivalent_one} follows by comparing~\eqref{P_DT_ub_rational} and~\eqref{P_DT_lb} and than by using~\eqref{def_T_zero_D_rational}.


We readily infer~\eqref{Q_double_estimation} by using part~\ref{annexe_P7_part_two} of Theorem~\ref{annexe_P7} (and taking into account as well the double bound~\eqref{lemma_db_p_T_result}).

We conclude by repeating arguments from the end of the proof of Proposition~\ref{existence_P_DT}.
First, in view of~\eqref{annexe_defPT}, the polynomial $P_{D_0,D,T}$ satisfies
\begin{equation} \label{P_DT_equality}
R_{D_0,D,T}(\ul{x})=\left(\prod_{i=0}^{T-1}q(p^{[i]}(y))\right)^{nD}\cdot F(p^{[T]}(y)).
\end{equation}
Secondly, the double bound~\eqref{lemma_db_p_T_result} implies, for some constants $\tilde{c}_9,\tilde{c}_{10}>0$,
\begin{equation} \label{P_DT_auxilliary_db}
\exp(-\tilde{c}_9D_0D^{t}\delta^T)\leq\prod_{i=0}^{T-1}q(p^{[i]}(y))\leq \tilde{c}_{10}^{nD}.
\end{equation}
Finally, combine~\eqref{Q_double_estimation}, \eqref{P_DT_equality} and~\eqref{P_DT_auxilliary_db} to find~\eqref{Q_double_estimation_statement}.
This completes the proof of the proposition.
\end{proof}

\section{Proofs of Main Theorems} \label{section_Proofs}

In this section we give proofs of our main results announced in the Introduction, Theorems~\ref{annexe_theo_ia1}, \ref{annexe_theo_ia2} and~\ref{annexe_theo_2}.

Recall that we denote by $K$ a number field containing all the coefficients of Taylor expansions of $f_1(z),\dots,f_n(z)$ at $z=0$, all the coefficients of polynomials involved in the system~\eqref{systeme_1} (that is, all the coefficients of entries of matrices $A(z)$ and $B(z)$ and all the coefficients of $p(z)$ and $a(z)$). Also we assume $y\in K$ in the proofs of Theorems~\ref{annexe_theo_ia1} and~\ref{annexe_theo_2}, that is whenever we assume $y\in\ol{\mqq}$ we automatically assume $y\in K$.

{\noindent \it Proof of Theorem~\ref{annexe_theo_ia1}.} 
Let $W\subset\mpp^n_{\mqq}$ be a variety of dimension $k<t+1-\frac{\log d}{\log\delta}$. 
\hidden{\begin{multline*}
    D_0=\Bigg[c\cdot\max\Bigg(\left(\frac{t(W)}{\log(t(W))+1}\right)^{\frac1{n+1-k}},\\ \left(d(W)(\log(t(W))+1)^{\frac{\log d}{\log\delta}-1}\right)^{\frac{1}{n+1-k-\frac{\log d}{\log\delta}}}\Bigg)\Bigg],
\end{multline*}
o\`u $c$ d\'esigne une constante suffisamment grande.}
Let $c>0$ and $c_0>0$ be sufficiently big constants, to be fixed later. Define
\begin{eqnarray}
  D_0&:=&\left\lceil c_0\frac{h(W)}{\max\left(\log h(W),2\right)}  \right\rceil, \label{annexe_theo_ia1_defD0}
\\
  D' &:=& c \deg(W)^{\frac{1}{t-k+1-\frac{\log d}{\log\delta}}}\left(\log D_0\right)^{\frac{\log d-\log\delta}{(t-k+1-\frac{\log d}{\log\delta})\log\delta}}, \label{annexe_theo_ia1_defD}
\\
  T &:=& \lceil T_1(D_0,D')\rceil. \label{annexe_theo_ia1_defT}
\end{eqnarray}
\hidden{
For each value $c>0$ we define the constant $c_0\geq 1$ to be the minimal real number $\geq 1$ ensuring 
\begin{equation} \label{annexe_theo_ia1_def_c0}
D_0\geq D'^{\frac{\log d}{\log\delta}}.
\end{equation}
Indeed, such a (finite) value $c_0\geq 1$ clearly exists, because if we substitute the definition of $D'$, \eqref{annexe_theo_ia1_defD}, into the right hand side of~\eqref{annexe_theo_ia1_def_c0}, then we find the inequality where the left hand side is simply $D_0$ and the right hand side is a power of a logarithm of $D_0$, so this inequality holds true for $D_0$ sufficiently large, which can clearly be ensured by choosing $c_0\geq 1$ to be sufficiently large.
}


We apply Proposition~\ref{existence_P_DT}, deducing the existence of polynomials $P_{D_0,D,T}$ satisfying~\eqref{ub_P_DT_deg_z}, \eqref{ub_P_DT_deg_X}, \eqref{ub_P_DT_L} and~\eqref{Q_double_estimation_statement}. We readily find that polynomials $\tilde{P}_{D_0,D,T}(X_1,\dots,X_n)\in K[X_1,\dots,X_n]$,
$$
\tilde{P}_{D_0,D,T}(X_1,\dots,X_n):=P_{D_0,D,T}(y,X_1,\dots,X_n)
$$
verify
\begin{eqnarray} \label{ub_tildeP_DT_deg_X}
\deg_{\ul{X}}\tilde{P}_{D_0,D,T}&<&n D,
\\
L(\tilde{P}_{D_0,D,T})&\leq&\exp(c_{11}D_0(d^T+\log D_0)), \label{ub_tildeP_DT_L}
\\
\exp(-c_{12}D_0D^{t}\delta^T)&\leq&|\tilde{P}_{D_0,D,T}(\ul{x})|\leq\exp(-c_{13}D_0D^{t}\delta^T),  \label{db_tildeP_DT_value}
\end{eqnarray}
where $c_{12}\geq c_{13}>0$ are some constants. 

Note that definitions~\eqref{annexe_defT} and~\eqref{annexe_theo_ia1_defT} readily imply that, for $c$ big enough, there exist constants $c_9, c_{10}>0$ such that
\begin{equation} \label{db_delta_T}
c_9 D'\log D_0\leq \delta^T\leq c_{10}D'\log D_0.
\end{equation}

The upper bound~\eqref{ub_tildeP_DT_L} together with the inequality~\eqref{annexe_defT} 
imply that for $y\in K$ the length of $P_{D_0,D,T}(y)\in K[X_1,\dots, X_n]$ is upper bounded by $\exp(c_6'''D_0d^T)$, for all $D_0,T$ defined by~\eqref{annexe_theo_ia1_defD0} and~\eqref{annexe_theo_ia1_defT} and $1\leq D\leq D'$ (where the constant $c_6'''$ depends on $y$ but is independent of $D_0$, $D'$ and $T$).

\hidden{
Let us denote $y_1=y$, and let us denote by $y_2$, ..., $y_w$ the set of conjugates of $y$ (so we assume that $y$ has degree $w$ over $\mqq$). Define
\begin{equation*}
    \tilde{P}_{D,T}(X_1,\dots,X_n)\eqdef{\rm den}(y)^{w\cdot\deg_{\b{z}}(P_{D,T})}\times\prod_{i=1}^wP_{D,T}(y_i,X_1,\dots,X_n),
\end{equation*}
this is a polynomial with integer coefficients, of degree $\leq w\cdot D$ in $X_1,\dots,X_n$ and of length
\begin{equation} \label{annexe_theo_ia1_majLtildeP}
    L(\tilde{P}_{D,T})\leq\exp\left(\tilde{c}_6Dd^T\right).
\end{equation}
}
We apply Theorem~\ref{CplM} 
to the point $\ul{x}'=\left(f_1(y),\dots,f_n(y)\right)\in\mcc^{n}$ with
\begin{equation} \label{annexe_theo_ia1_parameters}
\begin{aligned}
    &m=n,\quad 0\leq k=\dim W < t+1-\frac{\log d}{\log\delta}, 
\quad \lambda=2,
\\
&\delta_1=\dots=\delta_n=D',
    \quad\tau=\tilde{c}_6D_0d^T, \quad
    \sigma=\frac{4c_{12}}{c_{13}},\quad U=\frac12 c_{13} D_0 D'^{t}\delta^T,
\end{aligned}
\end{equation}
where $\tilde{c}_6=\max\left(c_6''',c_{13}/2\right)$.

We are going to verify hypothesis of Theorem~\ref{CplM}. First, note that inequality $U>\tau$ readily follows from~\eqref{annexe_theo_ia1_defD}, 
by choosing $c$ sufficiently large (in this case, to satisfy $c>\frac{2\tilde{c}_6c_{10}^{\frac{\log d}{\log\delta}}}{c_{13}}$).

For every real number $s$ verifying
\begin{equation} \label{annexe_restriction_s}
    2\tau<s\leq U,
\end{equation}
we define $Q_s(X_1,\dots,X_n)=\tilde{P}_{D_0,D_s,T}(X_1,\dots,X_n)$ with $D_s=\left[\left(\frac{2s}{c_{13}D_0\delta^T}\right)^{1/t}\right]$, $D_0$ and $T$ defined by~\eqref{annexe_theo_ia1_defD0} and~(\ref{annexe_theo_ia1_defT}) respectively.
The constraints~(\ref{annexe_restriction_s}) imply
\begin{equation} \label{Ds_double_ie}
    \left[\left(\frac{4\tilde{c}_6}{c_{13}}\left(\frac{d}{\delta}\right)^T\right)^{1/t}\right] \leq D_s=\left[\left(\frac{2s}{c_{13}D_0\delta^T}\right)^{1/t}\right] \leq D',
\end{equation}
so the quantity $e(s)$ of Theorem~\ref{CplM} satisfies
\begin{equation*}
    e(s)=\left[\frac{D_s}{D'}\right]+1\leq 2=\lambda.
\end{equation*}
We proceed with verification of condition~\eqref{CplM_approximation}. First, note that $\prod_{i=1}^{n+1}\left(1+|x_i|^2\right)$, where $x_i=f_i(y)$, $i=1,\dots,n$, is a constant. Also, recall that $\deg_{\ul{X}}P_{D_0,D_s,T}=Q_s<n D_s$ by~\eqref{ub_P_DT_deg_X}. So, for a constant $c'$, we have
\begin{equation} \label{annexe_degX_petit}
    \prod_{i=1}^{n+1}\left(1+|x_i|^2\right)^{\frac12\deg_{X_i}Q_{s}}\leq\exp\left(c'D_s\right).
\end{equation}

Further, bigger the constant $c_{12}$ is in the lower bound in~\eqref{db_tildeP_DT_value}, weaker is this lower bound. So we can assume, without loss of generality, that the constants $c_{12}$ is sufficiently big, for instance, we can assume $c_{12}>2c'$. 
Then, \eqref{annexe_degX_petit} implies
\begin{equation} \label{annexe_degX_petit_c_7}
    \prod_{i=1}^{n+1}\left(1+|x_i|^2\right)^{\frac12\deg_{X_i}Q_{s}}<e^{\frac12c_{12}D_s}.
\end{equation}
Condition~\eqref{CplM_approximation} readily follows, for $c>0$ sufficiently big, from~\eqref{annexe_degX_petit_c_7}, \eqref{db_tildeP_DT_value} and the remark that $\sigma/2\geq 2\geq e(s)$.

We verify with the definition of $D_0$, $D'$ and $T$ the hypothesis~\eqref{CplM_VarietyRestriction}, that is that for $c$ big enough we have
\begin{equation} \label{annexe_theo_ia1_VarietyRestriction}
    [K:\mqq]\cdot 3\lambda^{k+1}D'^k\left(D' t(W)+(k+1)\tau d(W)\right)
    \leq U/\left(n\sigma\right)^{k+1}.
\end{equation}
Indeed, by taking $c$ big enough we can ignore the constant factors $[K:\mqq]\cdot 3\lambda^{k+1}$ in the left hand side and $\left(n\sigma\right)^{k+1}$ in the right hand side. Taking this into account, than substituting parameters~\eqref{annexe_theo_ia1_parameters}, using~\eqref{db_delta_T} and unwinding definition of $t(W)$, see~\eqref{def_tV}, we find that the inequality~\eqref{annexe_theo_ia1_VarietyRestriction} boils down to the verification, for $c>0$ big enough, of the following system of inequalities:
\begin{eqnarray}
D'^{k+1}h(W)&\leq& D_0 D'^{t+1} \log D_0, \label{final_system_ie1}
\\
D'^k (n+1)\log(n+1)\deg(W) &\leq& D_0 D'^{t+1} \log D_0, \label{final_system_ie2}
\\
D_0D'^{k+\frac{\log d}{\log\delta}} \left(\log D_0\right)^{\frac{\log d}{\log\delta}} \deg(W)&\leq& D_0 D'^{t+1} \log D_0, \label{final_system_ie3}
\end{eqnarray}
where $D_0$ and $D'$ are defined by~\eqref{annexe_theo_ia1_defD0} and~\eqref{annexe_theo_ia1_defD} respectively. It is easy to verify that, for $c$ big enough, the inequality~\eqref{final_system_ie2} follows from~\eqref{final_system_ie3}. So we need to verify only inequalities~\eqref{final_system_ie1} and~\eqref{final_system_ie3}.

The inequality~\eqref{final_system_ie1} readily follows from $k<t$ and the definition~\eqref{annexe_theo_ia1_defD0}. Indeed, the definition of $D_0$ implies that for $c_0>0$ sufficiently big we have $D_0\log D_0 \geq h(W)$.

To verify~\eqref{final_system_ie3}, we rewrite it in the form
$$
\left(\log D_0\right)^{\frac{\log d}{\log\delta}-1} \deg(W)\leq D'^{t-k+1-\frac{\log d}{\log\delta}},
$$
which then readily follows by~\eqref{annexe_theo_ia1_defD}.

So, \eqref{annexe_theo_ia1_VarietyRestriction} indeed holds true for $c>0$ big enough and we can apply Theorem~\ref{CplM}.
Using this theorem, we find, for constants $C_1,C>0$,
\begin{equation*}
\begin{aligned}
    \log\Dist(\ul{x},W) &\geq -U=-C_1D_0D^{t+1}\log D_0
\\
&\geq -Ch(W)\left(\log h(W)\right)^{\frac{(t+1)\left(\frac{\log d}{\log\delta}-1\right)}{t-k+1-\frac{\log d}{\log\delta}}}\left(\deg(W)\right)^{\frac{t+1}{t-k+1-\frac{\log d}{\log\delta}}}.
\end{aligned}
\end{equation*}
It completes the proof of Theorem~\ref{annexe_theo_ia1}.\ep

{\noindent \it Proof of Theorem~\ref{annexe_theo_ia2}.} We use the same method as in the proof of Theorem~\ref{annexe_theo_ia1}. 
Because of this similarity, we omit some details in the proof below, referring the reader to the corresponding places in the proof of Theorem~\ref{annexe_theo_ia1} for some more explanations.

Let $W\subset\mpp^{n+1}_{\mqq}$ be a variety of dimension $k<n+1-2\frac{\log d}{\log\delta}$ 
and let $\varepsilon>0$ be a number such that $k<n+1-2\frac{\log d}{\log\delta}-\varepsilon$
Define
\begin{eqnarray}
  D' &:=& c  \cdot\max\left(\left(\deg(W)\right)^{\frac1{t-k+1-2\frac{\log d}{\log\delta}-\varepsilon}},h(W)^{\frac1{t-k+2-\frac{\log d}{\log\delta}-\varepsilon}}\right), \label{annexe_theo_ia1_defD_complex}
\\
  T &:=& \lceil T_1(D',D')\rceil. \label{annexe_theo_ia1_defT_complex}
\end{eqnarray}
where $c$ denotes a sufficiently big constant. We apply Proposition~\ref{existence_P_DT} to deduce, for every $D\leq D'$, the existence of polynomials $P_{D,T}$ satisfying~\eqref{ub_P_DT_deg_z}, \eqref{ub_P_DT_deg_X}, \eqref{ub_P_DT_L} and~\eqref{Q_double_estimation_statement} (with the notation $D_0:=D$).

Note that definitions~\eqref{annexe_defT} and~\eqref{annexe_theo_ia1_defT} readily imply that, for $c$ big enough, there exist constants $c_9, c_{10}>0$ such that
\begin{equation} \label{db_delta_T_complex}
c_9 D'\log D'\leq \delta^T\leq c_{10}D'\log D'.
\end{equation}

We apply Theorem~\ref{CplM} to the point $\ul{x}=\left(y,f_1(y),\dots,f_n(y)\right)\in\mcc^{n+1}$ and with the following set of parameters
\begin{multline*}
    m=n+1,\quad 0\leq k=\dim W < t+1-2\frac{\log d}{\log\delta},\quad \lambda=2,\\ \delta_1=\tau=\tilde{c}_6D'd^T,
    \delta_2=\cdots=\delta_{n+1}=D',\quad
    \sigma=4c_7/c_8,\quad U=\frac12c_8D'^{t+1}\delta^T 
\end{multline*}
where $\tilde{c}_6=\max(c_5,c_6)$.
For all real $s$ verifying
\begin{equation} \label{annexe_restriction_s_1}
    \tau\lambda<s\leq U,
\end{equation}
we define $R_s=P_{D_s,T}$ with $D_s=\left[\left(\frac{2s}{c_8\delta^T}\right)^{1/(t+1)}\right]$. 
The double bound~(\ref{annexe_restriction_s}) implies
\begin{equation} \label{annexe_theo_ia_2_Ds_double_ie_complex}
    \left[\left(\frac{4c_6}{c_8}D'\left(\frac{d}{\delta}\right)^T\right)^{\frac{1}{t+1}}\right] \leq D_s=\left[\left(\frac{2s}{c_8\delta^T}\right)^{1/(t+1)}\right] \leq D'.
\end{equation}
Hence the quantity $e(s)$ in the statement of Theorem~\ref{CplM} satisfies
\begin{equation*}
    e(s)\leq2=\lambda
\end{equation*}
and we readily infer~(\ref{CplM_approximation}) (see the proof of Theorem~\ref{annexe_theo_ia1}, in particular~\eqref{annexe_degX_petit} and~\eqref{annexe_degX_petit_c_7} for some more details).

Similarly to the proof of Theorem~\ref{annexe_theo_ia1}, verification of hypothesis~\eqref{CplM_VarietyRestriction} is equivalent, up to adjusting the constant $c>0$ in the definition of $D'$, to the following two inequalities
\begin{eqnarray}
D'^{k+\frac{\log d}{\log\delta}}\left(\log D'\right)^{\frac{\log d}{\log\delta}}h(W)&\leq& D'^{t+2}\log D', \label{final_system_ie1_complex}
\\
D'^{k+1+2\frac{\log d}{\log\delta}}\left(\log D'\right)^{2\frac{\log d}{\log\delta}} \deg(W)&\leq& D'^{t+2}\log D'. \label{final_system_ie3_complex}
\end{eqnarray}
Note that by choosing the constant $c$ in~\eqref{annexe_theo_ia1_defD_complex} sufficiently big we can ensure
$$
D'^{\varepsilon}\geq\left(\log D'\right)^{2\frac{\log d}{\log\delta}}\geq\left(\log D'\right)^{\frac{\log d}{\log\delta}-1}.
$$
Then inequalities~\eqref{final_system_ie1_complex} and~\eqref{final_system_ie3_complex} readily follow from~\eqref{annexe_theo_ia1_defD_complex}.

So, all the hypothesis of Theorem~\ref{CplM} are verified and we infer with this theorem, for a constant $C>0$,
\begin{multline*}
    \log\Dist(\ul{x},W)\geq -U=-\frac12c_8D'^{t+1}\delta^T
\\
\geq -C\max\left(\left(\deg(W)\right)^{\frac{t+2+\varepsilon}{t-k+1-2\frac{\log d}{\log\delta}-\varepsilon}},h(W)^{\frac{t+2+\varepsilon}{t-k+2-\frac{\log d}{\log\delta}-\varepsilon}}\right).
\end{multline*}
This completes the proof of Theorem~\ref{annexe_theo_ia2}. \ep

{\noindent \it Proof of Theorem~\ref{annexe_theo_2}.} We use the same method of proof as for Theorems~\ref{annexe_theo_ia1} and~\ref{annexe_theo_ia2}. The only modifications we need is to replace Proposition~\ref{existence_P_DT}  by Proposition~\ref{existence_P_DT_rational_p}, adjust the values of parameters according to bounds given by this new  proposition and verify, once again, the hypothesis of Theorem~\ref{CplM}.

So, let $W\subset\mpp^n_{\mqq}$ be a variety of dimension $k<t\left(2-\frac{\log d}{\log\delta}\right)$. 
Define
\begin{eqnarray}
  D_0&:=&\left\lceil h(W)^{1/2}  \right\rceil, \label{annexe_theo_ia1_defD0_rational}
\\
  D' &:=& c \left(\deg(W)D_0^{\left(\frac{\log d}{\log\delta}-1\right)}\right)^{\frac1{t\left(2-\frac{\log d}{\log\delta}\right)-k}}, \label{annexe_theo_ia1_defD_rational}
\\
  T &:=& \lceil T_2(D_0,D')\rceil. \label{annexe_theo_ia1_defT_rational}
\end{eqnarray}
where $c$ denotes a sufficiently big constant.

Note that definitions~\eqref{annexe_defT_rational} and~\eqref{annexe_theo_ia1_defT_rational} readily imply that, for $c$ big enough, there exist constants $c_{16}, c_{17}>0$ such that
\begin{equation} \label{db_delta_T}
c_{16}D_0 D'^t\leq \delta^T\leq c_{17}D_0 D'^t.
\end{equation}

By Proposition~\ref{existence_P_DT_rational_p}, for all $D\leq D'$  there exist polynomials $R_{D_0,D,T}\in K[\b{z},X_1,\dots,X_n]$ satisfying~\eqref{annexe_theo_ia1_defD0_rational}, \eqref{annexe_theo_ia1_defD_rational}, \eqref{annexe_theo_ia1_defT_rational} and~\eqref{Q_double_estimation_statement}. Then, the polynomials $\tilde{R}_{D_0,D,T}\in K[X_1,\dots,X_n]$ defined by
$$
\tilde{R}_{D_0,D,T}(X_1,\dots,X_n):=R_{D_0,D,T}(y,X_1,\dots,X_n).
$$
satisfy
\begin{eqnarray} \label{ub_tildeP_DT_deg_X_rational}
\deg_{\ul{X}}\tilde{R}_{D_0,D,T}&<&n D,
\\
L(\tilde{R}_{D_0,D,T})&\leq&\exp(c'_{18}D_0(D^{t}+d^T))\leq\exp(c_{18}D_0^{1+\frac{\log d}{\log\delta}}D'^{\frac{\log d}{\log\delta}t}), \label{ub_tildeP_DT_L_rational}
\\
\exp(-c_{19}D_0^2D^{t}D'^t)&\leq&|\tilde{R}_{D_0,D,T}(\ul{x})|\leq\exp(-c_{20}D_0^2D^{t}D'^t),  \label{db_tildeP_DT_value_rational}
\end{eqnarray}
where $c_{19}\geq c_{20}>0$ are some constants.

We apply Theorem~\ref{CplM} to the point $\ul{x}=\left(f_1(y),\dots,f_n(y)\right)\in\mcc^{n}$ with
\begin{multline} \label{theo_2_choice_of_parameters}
    m=n,\quad 0\leq k=\dim W < t\left(2-\frac{\log d}{\log\delta}\right),\quad \lambda=2,
\\
\tau=\tilde{c}_{18}D_0^{1+\frac{\log d}{\log\delta}}D'^{\frac{\log d}{\log\delta}t},
    \delta_1=\delta_2=\cdots=\delta_{n}=D,
\\
    \sigma=\frac{4c_{19}}{c_{20}},\quad U=\frac12c_{20}D_0^2D'^{2t},
\end{multline}
where $\tilde{c}_{18}=\max(c_{18},c_{19})$.


We readily verify the hypothesis of Theorem~\ref{CplM} for the choice of parameters~\eqref{annexe_theo_ia1_defD0_rational}, \eqref{annexe_theo_ia1_defD_rational}, \eqref{annexe_theo_ia1_defT_rational} and~\eqref{theo_2_choice_of_parameters}. So, the hypothesis $\tau\lambda<U$ straightforwardly follows, for $c>0$ sufficiently large, from~\eqref{annexe_theo_ia1_defD_rational} and~\eqref{theo_2_choice_of_parameters}.

Further, for every real $s$ verifying
\begin{equation} \label{annexe_theo_2_restriction_s_1}
    \tau\lambda<s\leq U
\end{equation}
we define $R_s(X_1,\dots,X_n)=\tilde{R}_{D_0,D_s,T}(X_1,\dots,X_n)\in K[X_1,\dots,X_n]$, where $D_s=\left[\left(\frac{2s}{c_{20}D_0^2D'^t}\right)^{1/t}\right]$.

The constraints~\eqref{annexe_theo_2_restriction_s_1} imply
\begin{equation} \label{annexe_theo_2_Ds_double_ie}
    \left[\left(\frac{4\tilde{c}_{18}}{c_{20}\left(D_0D'^t\right)^{1-\frac{\log d}{\log\delta}}}\right)^{1/t}\right] \leq D_s=\left[\left(\frac{2s}{c_{20}D_0^2D'^t}\right)^{1/t}\right] \leq D',
\end{equation}
so the quantity $e(s)$ in Theorem~\ref{CplM} satisfies
\begin{equation*}
    e(s)\leq\lambda.
\end{equation*}
Moreover, the lower bound in~\eqref{annexe_theo_2_Ds_double_ie} imply that
$$
D_0^2D_s^tD'^t\geq\frac{4\tilde{c}_{18}}{c_{20}}D_0^{1+\frac{\log d}{\log\delta}}D'^{\frac{\log d}{\log\delta}t}\geq 2 e(s)\tau.
$$
Also note that, with the choice of parameters~\eqref{theo_2_choice_of_parameters}, we clearly have $\tau>\delta_1+\dots+\delta_m$ for $c>0$ big enough.
This ensures~(\ref{CplM_approximation}) of Theorem~\ref{CplM}.

Verification of hypothesis~\eqref{CplM_VarietyRestriction} boils down to verifying, for $c>0$ large enough, the following two inequalities (see the proof of Theorem~\ref{annexe_theo_ia1} for more details on this reduction):
\begin{eqnarray}
D'^{k+1}h(W)&\leq& D_0^2 D'^{2t}, \label{final_system_ie1_rational}
\\
D_0^{\frac{\log d}{\log\delta}+1}D'^{k+\frac{\log d}{\log\delta}t} \deg(W)&\leq& D_0^2 D'^{2t}. \label{final_system_ie3_rational}
\end{eqnarray}
Than, the inequality~\eqref{final_system_ie1_rational} readily follows from~\eqref{annexe_theo_ia1_defD0_rational} and~\eqref{final_system_ie3_rational} follows from~\eqref{annexe_theo_ia1_defD_rational}.

So, all the hypothesis of Theorem~\ref{CplM} are verified and we infer with this theorem, for a constant $C>0$,
\begin{multline*}
    \log\Dist(\ul{x},W)\geq -U=-CD_0^2D'^{2t}
\\
=-Ch(W)^{1+\frac{\log d-\log\delta}{(2t-k)\log\delta-t\log d}t}\deg(W)^{\frac{2t}{t\left(2-\frac{\log d}{\log\delta}\right)-k}}.
\end{multline*}
This proves~\eqref{annexe_theo_2_result}.
\ep



\hidden{
\begin{center}%
          {\bfseries Acknowledgement\vspace{-.5em}}%
\end{center}%
The author would like to express his profound gratitude to Patrice \textsc{Philippon} for his support throughout  this research. 
}

{\small



\begin{thebibliography}{0}


\bibitem{AB2011} B.~Adamczewski, Y.~Bugeaud, ``Nombres r\'eels de complexit\'e sous-lin\'eaire: mesures
d'irrationalit\'e et de transcendance'', J. Reine Angew. Math., 658:65-98, 2011.

\bibitem{AC2006} B.~Adamczewski, J.~Cassaigne, ``Diophantine properties of real numbers generated by finite automata'', Compos. Math., 142(6):1351-1372, 2006.


\bibitem{AS2003} J.-P.~ Allouche, J.~Shallit, ``Automatic sequences'', Cambridge University Press, Cambridge, 2003.

\bibitem{A1991} M.~Amou, ``Algebraic independence of the values of certain functions at a transcendental number'', Acta Arithmetica 59.1 (1991): 71-82.

\bibitem{B1994} P.-G.~Becker, ``Transcendence measures for the values of generalized Mahler functions in arbitrary
characteristic'', Publ. Math. Debrecen 45 (1994), 269-282.

\bibitem{B1994_k} P.-G.~Becker, ``$k$-regular power series and Mahler-type functional equations'', J. Number Theory, 49(3), 1994, 269–286

\bibitem{BBC2013} J.~ Bell, Y.~Bugeaud, M.~Coons, ``Diophantine approximation of Mahler numbers'', Proc. London Math. Soc. (2015) 110 (5): 1157-1206, doi: 10.1112/plms/pdv016

\bibitem{Ch1980} G.~V.~Chudnovsky, ``Measures of irrationality, transcendence and algebraic independence. Recent progress'': Journ\'ees Arithm\'etiques~1980 (J.~Armitage, ed.), Cambridge Univ. Press, 1982, 11-82.

\bibitem{Cobham1968} A.~Cobham ``On the Hartmanis-Stearns problem for a class of tag machines'' in Switching and Automata Theory, 1966., IEEE Conference Record of Seventh Annual Symposium, 1968.

\bibitem{DHR2015} Th.~Dreyfus, Ch.~Hardouin, J.~Roques, ``Hypertranscendence of solutions of Mahler equations'', arXiv:1507.03361 .

\bibitem{HS1965} J.~Hartmanis and R.~E.~Stearns, ``On the Computational Complexity of Algorithms'', Transactions of the American Mathematical Society, Vol. 117 (May, 1965), pp. 285-306.

\bibitem{Kok1939} J.~F.~Koksma, ``\"Uber die Mahlersche Klasseneinteilung der transzendenten Zahlen und die Approximation komplexer Zahlen durch algebraische Zahlen'', Monatsh. Math. Phys., 48:176-189, 1939.

\bibitem{Kubota1977} K.~K.~Kubota, ``On the algebraic independence of holomorphic solutions of certain functional equations and their values'', Math.Ann. 227 (1977), 9-50.

\bibitem{J1996} C.~Jadot, ``Crit\`eres pour l'ind\'ependance alg\'ebrique et lin\'eaire'', th\`ese de doctorat de l'Universit\'e Paris 6, 1996. Disponible \`a hal.archives-ouvertes.fr.

\bibitem{Mah1929} K.~Mahler, ``Arithmetische Eigenschaften der L\"osungen einer Klasse von
Funktionalgleichungen'', Math. Ann. 101 (1929), 342-366.

\bibitem{LvdP1982} J.~H.~Loxton and A.~J.~van~der~Poorten, ``Arithmetic properties of the solutions
of a class of functional equations'', J. reine angew. Math. 330 (1982), 159-172.

\bibitem{LvdP1988} J.~H.~Loxton and A.~J.~van~der~Poorten, ``Arithmetic properties of automata:
regular sequences'', J. reine angew. Math. 392 (1988), 57-610.


\bibitem{Mah1930} K.~Mahler, ``Arithmetische Eigenschaften einer Klasse transzendentaltranszendenter
Funktionen'', Math. Z. 32 (1930), 545-585.

\bibitem{Mah1930_2} K.~Mahler, ``\"Uber das Verschwinden von Potenzreihen mehrerer Ver\"anderlichen
in speziellen Punktfolgen'', Math. Ann. 103 (1930), 573-587.

\bibitem{Mah1932} K.~Mahler, ``Zur Approximation der Exponentialfunktionen und des Logarithmus'', I, II. J. reine angew. Math., 166:118-150, 1932.

\bibitem{NP} Yu.~Nesterenko, P.~Philippon (eds.), ``Introduction to Algebraic Independence Theory'', Vol.~1752, 2001, Springer.

\bibitem{Ni1986} K.~Nishioka, ``Algebraic independence of certain power series of algebraic numbers'', J. Number Theory 23 (1986), 353-364.

\bibitem{Ni1996} K.~Nishioka, ``Mahler Functions and Transcendence'', Lecture Notes in Math. 1631, Springer,
1996.

\bibitem{Pellarin2010} F.~Pellarin, ``An introduction to Mahler's method for
transcendence and algebraic independence'', preprint, 2010. Disponible at http://hal.archives-ouvertes.fr/hal-00481912/fr/

\bibitem{PP_AH} P.~Philippon, ``Sur des hauteurs alternatives I'', Math. Ann. 289 (1991),255-283; II, Ann. Inst. Fourrier (Grenoble) 44/4 (1994), 1043-1065; III, J. Math. Pures Appl. 74/4 (1995), 343-365.

\bibitem{PP1997} P.~Philippon, ``Une approche m\'ethodique pour la transcendance et l'ind\'ependance alg\'ebrique de valeurs de fonctions analytiques'', J. Number Theory { 64} (1997) 291-338.

\bibitem{PP_KF} P.~Philippon, ``Ind\'ependance alg\'ebrique et $K$-fonctions'', J. reine angew. Math. 497 (1998), 1-15.

\bibitem{PP_2011} P.~Philippon, ``Some aspects of Mahler's method'', Manuscript, 2011

\bibitem{PP2015} P.~Philippon, ``Groupes de Galois et nombres automatiques'', preprint, 2015. arXiv:1502.00942 .

\bibitem{Ritt1922} J.~F.~Ritt, ``Permutable Rational Functions'', Transactions of the AMS, vol. 23 (1922), 399-448.

\bibitem{Ta1991} J.~Tamura, ``Symmetric continued fractions related to certain series'', J. Number Theory 38 (1991), 251-264.

\bibitem{ThTopfer1995} Th.~T\"opfer, ``Algebraic independence of the values of generalized Mahler functions'', Acta
Arithmetica, LXX.2 (1995).

\bibitem{EZ2010} E.~Zorin, ``Lemmes de z\'eros et relations fonctionnelles'', th\`ese de doctorat de l'Universit\'e Paris 6, 2010. Accessible at http://tel.archives-ouvertes.fr/tel-00558073/fr/

\bibitem{EZ2011_2} E.~Zorin, ``New results on algebraic independence with Mahler's method'', Comptes Rendus Acad. Sci. Paris, Ser. I 349 (2011) 607-610.

\bibitem{EZ2013} E.~Zorin, ``Zero Order Estimates for Analytic Functions'', International Journal of Number Theory, 9(2) (2013), 1-60.

\bibitem{EZ2013_2} E.~Zorin, ``Multiplicity Estimates for Algebraically Dependent Analytic Functions'', Proceedings of the London Mathematical Society, Vol. 108, No. 4, 04.2014, p. 989-1029.
\end{thebibliography}

\def\cprime{$'$} \def\cprime{$'$} \def\cprime{$'$} \def\cprime{$'$}
  \def\cprime{$'$} \def\cprime{$'$} \def\cprime{$'$} \def\cprime{$'$}
  \def\cprime{$'$} \def\cprime{$'$} \def\cprime{$'$} \def\cprime{$'$}
  \def\cprime{$'$} \def\cprime{$'$} \def\cprime{$'$} \def\cprime{$'$}
  \def\polhk#1{\setbox0=\hbox{#1}{\ooalign{\hidewidth
  \lower1.5ex\hbox{`}\hidewidth\crcr\unhbox0}}} \def\cprime{$'$}
  \def\cprime{$'$}



}

\end{document}